\documentclass{article} 
\def\date{31.1.2011}
\usepackage{amssymb}
\usepackage{amsmath}

\newcommand{\1}{\mathbf{1}}
\input{liemacs.sty}

\newcommand{\bn}{{\bf n}} 
\newcommand{\co}{\mathop{{\rm co}}\nolimits}
\newcommand{\Part}{\mathop{{\rm Part}}\nolimits}
\newcommand{\Tab}{\mathop{{\rm Tab}}\nolimits}

\newcommand{\Inm}{I^{\bf n}}
\newcommand{\sot}{\textsc{so}}

\newcommand{\wot}{\textsc{wo}}
\newcommand{\trace}{{\rm Tr}\,}

\newcommand{\Dent}{{\rm Dent}}
\newcommand{\StrExp}{{\rm StrExp}}

\newcommand{\cX}{{\mathcal X}}

\begin{document} 



\title{Schur--Weyl Theory for $C^*$-algebras} 
\author{Daniel Belti\c t\u a\footnote{Institute of Mathematics ``Simion
Stoilow'' of the Romanian Academy, 
P.O. Box 1-764, Bucharest, Romania;  
email: \texttt{Daniel.Beltita@imar.ro}}\ \ 
and Karl-Hermann Neeb
\footnote{Department Mathematik, FAU Erlangen-N\"urnberg, 
Bismarckstrasse 1$^{\scriptscriptstyle 1} \hskip-3pt / \hskip-2pt _{\scriptscriptstyle 2}$, 
91054-Erlangen, Germany;
email: \texttt{neeb@mi.uni-erlangen.de}}}

\maketitle

\begin{abstract} To each irreducible infinite dimensional representation 
$(\pi,\cH)$ of a $C^*$-algebra $\cA$, we associate 
a collection of irreducible norm-continuous unitary representations 
$\pi_{\lambda}^\cA$ of its unitary group $\U(\cA)$, whose equivalence 
classes are parameterized by highest weights in the same 
way as the irreducible bounded unitary representations 
of the group $\U_\infty(\cH) = \U(\cH) \cap (\1 + K(\cH))$ are. 
These are precisely the representations 
arising in the decomposition of the tensor products 
$\cH^{\otimes n} \otimes (\cH^*)^{\otimes m}$ under $\U(\cA)$. 
We show that these representations can be realized by 
sections of holomorphic line bundles over homogeneous 
K\"ahler manifolds on which $\U(\cA)$ acts transitively and 
that the corresponding norm-closed momentum sets 
$I_{\pi_\lambda^\cA}^{\bf n} \subeq \fu(\cA)'$ 
distinguish inequivalent representations of this type. 
\end{abstract} 

\tableofcontents

\section*{Introduction} 

The relationship between properties of $C^*$-algebras 
and the representation theory of their unitary groups has come up 
in the literature ever since the classification in \cite{Se57} of the representations of the unitary group on a Hilbert space satisfying a certain positivity condition. 
Other classification results in this direction were later obtained;  
see for instance \cite{Ki73}, \cite{SV75}, \cite{Ol79}, \cite{Ol84}, \cite{Pi88}, \cite{Bo93} and the references therein. 
The broad issue of developing a systematic approach to the classification 
of representations of the unitary groups of 
$C^*$-algebras was raised in \cite{Pa83}, 
and this is precisely the problem which we address in the present work  
---by using different tools, however:  
namely, the momentum sets of unitary representations (cf.\ \cite{Ne00}, \cite{Ne08}, \cite{Ne09}, and \cite{Ne10}). 
For a Banach--Lie group $G$ and a norm-continuous unitary representation 
$\pi \: G \to \U(\cH)$, there are two variants of the momentum set. Both rely on the momentum map
\[ \Phi_\pi \: \bP(\cH) = \{ [v] = \C v \: 0\not= v \in \cH\} \to \g', \quad 
\Phi_\pi([v])(x) = \frac{\la -i\dd\pi(x)v,v\ra}{\la v,v\ra}.\] 
The larger one is the weak-$*$-closure $I_\pi$ of $\conv(\im(\Phi_\pi))$, and the smaller 
one is the norm-closure $I_\pi^{\bn}$ of this set. 
One connection between the present approach and the one of \cite{Pa83} will be established in Proposition~\ref{contr_SW} below. 
Before to describe the contents of our paper in some more detail, we shall briefly survey some progress that has been made so far. 
Specifically, the goal is to understand the irreducible bounded unitary representations 
of $\U(\cA)$ for a unital $C^*$-algebra $\cA$ in terms of 
irreducible algebra representations of $\cA$, 
whose classification is much better understood  
(as seen from \cite[Cor.~8]{GK60} or \cite[Thm.~1.1]{KOS03}, and also from \cite{Vo76}). 

If $\cA$ is commutative, then 
$\exp \: \fu(\cA) \to \U(\cA)$ induces a covering of the identity 
component, so that $\U(\cA)_0 \cong \fu(\cA)/\Gamma_\cA$, where 
$\Gamma_\cA = \ker(\exp_\cA\vert_{\fu(\cA)})$ is isomorphic to the discrete subgroup of 
$(\cA,+)$ generated by the projections. Since $\U(\cA)_0$ 
is divisible, we have 
$\U(\cA) \cong \U(\cA)_0 \times \pi_0(\U(\cA))$. 
This leads to the explicit description of the 
character group of $\U(\cA)$ in terms of data associated 
to $\cA$: 
\[ \U(\cA)\,\hat{} \cong 
\{ \alpha \in \fu(\cA)'\: \alpha(\Gamma_A) \subeq 2 \pi \Z\} 
\times \pi_0(\U(\cA))\,\hat{}.\] 
This description of the unitary dual of $\U(\cA)$ shows that, in 
general, it is not generated by the restrictions of 
algebra characters to $\cA$. This is only true if 
$\hat\cA = \Hom(\cA,\C)$ is totally disconnected (cf.\ \cite{Au93}). 

The next step 
is to consider algebras of the 
form $\cA = M_n(C(X))$, $X$ a compact space. 
For $X = \{*\}$ we have $\cA \cong M_n(\C)$, and 
the classical Schur--Weyl theory implies that 
all irreducible unitary representations of $\U(\cA) 
\cong \U_n(\C)$ occur in some finite tensor product of the identical 
representation on $\C^n$ and its dual. 
In general, $\U(\cA) = \U_n(C(X))$, 
and the irreducible representations of the identity component of this 
groups are described in \cite[Sect.~6]{NS10}. Actually 
\cite{NS10} deals with the group 
$\SU_n(\cA)_0$ corresponding to $\su_n(\cA)$, but since 
$\fu_n(\cA) \cong \fu(\cA) \oplus \su_n(\cA)$, 
all irreducible unitary representations of 
$\U_n(\cA)_0$ restrict to irreducible representations on 
$\SU_n(\cA)_0$ and the central subgroup $\U(\cA)_0\1$ is mapped 
into $\T\1$. In loc.\ cit.\ it is shown in particular 
that all irreducible representations 
factor through some homomorphism 
$\SU_n(C(X)) \to \SU_n(\C)^{|F|}$, induced by some 
evaluation map $\ev_F \: C(X) \to \C^F$, where $F \subeq X$ is a finite set. 
From the corresponding result for $\SU_n(\C)$, it thus follows 
that all irreducible bounded unitary representations of 
$\SU_n(\cA)$ embed in some tensor product of finitely 
many irreducible representations of $\cA$, which 
are given by evaluations $\ev_x \: M_n(C(X)) \to M_n(\C)$ in points of $X$ 
(\cite[Cor. 10.4.4]{Dix64}). 

In the present paper we analyze those irreducible representations 
of $\U(\cA)$ obtained from an irreducible algebra 
representation $(\pi, \cH)$ by decomposing 
the tensor products $\cH^{\otimes n} \otimes (\cH^*)^{\otimes m}$ 
under the action of $\U(\cA)$. 

The main method of our investigation consists in using the information encoded in the momentum sets of these representations. 
It is quite remarkable that, even in this situation of representations of infinite dimensional Lie groups, 
the data provided by the highest weight theory can be read off the extreme points of the momentum sets. 
Therefore, it eventually turns out that a mere look at suitable momentum sets in principle suffices for establishing whether two of these representations are 
equivalent. 
This approach to unitary representations of $\U(\cA)$ connects naturally with the 
classical $C^*$-algebraic methods. If $\pi$ is the restriction of an algebra 
representation $\pi \: \cA \to B(\cH)$, then its momentum set $I_\pi$ can be identified with the 
state space of the $C^*$-algebra $\pi(\cA) \cong \cA/\ker\pi$, hence 
contains no more information than 
the kernel of $\pi$ (\cite[Thm.~X.5.12]{Ne00}). 
The norm-closed momentum set $I_\pi^\bn$ can be identified with the 
normal state space of the von Neumann algebra $\pi(\cA)'' \cong \cA^{**}/\ker\pi^{**}$, where 
$\pi^{**} \: \cA^{**} \to B(\cH)$ is the associated representation of the enveloping $W^*$-algebra of 
$\cA$. Equivalently, $I_\pi^\bn$ determines the quasi-equivalence class of the algebra 
representation $\pi$. 

Here follows a more specific description of the contents of our paper. 
Section~\ref{Sect1} records general definitions, remarks and some auxiliary facts needed later on. 
We also introduce here the new notion of norm-closed momentum set for uniformly continuous representations of Banach--Lie groups. 

Section~\ref{Sect2} includes an updated exposition on the irreducible representations of the unitary group $\U_\infty(\cH)$ consisting of 
the unitary operators on the Hilbert space $\cH$ which are compact perturbations of the identity operator. 
The topics treated here are: highest weight theory, realization in tensor products, and geometric realizations on line bundles over flag manifolds. 
Moreover, we include here the first applications of the momentum sets in representation theory of unitary groups, by proving that the unitary irreducible representations of $\U_\infty(\cH)$ are distinguished from each other by their momentum sets. 

In Section~\ref{Sect3} we extend these results to irreducible representations of unitary groups of $C^*$-algebras. 
Thus, we obtain rather precise descriptions of the corresponding norm-closed momentum sets and show that the extreme points of any of these momentum sets is a coadjoint orbit, just as in the case of compact Lie groups. 
This fact plays a crucial role for the applications we make to the problem of deciding whether two representations are unitarily equivalent or not. 

In Section~\ref{Sect4} we use some basic notions of infinite dimensional convexity in order to get additional information on the extreme points of the momentum sets, by looking at the continuity points of the identity map from the norm topology on a momentum set to the weak-$*$-topology on the same set. 

The paper concludes by two appendices. 
The first of them records several auxiliary results of a rather algebraic nature needed in the main body of the paper. 
The second appendix is devoted to a discussion on topological properties and extreme points for the orbits of certain infinite Weyl groups; these properties are needed in Section~\ref{Sect3} and are also interesting on their own.

\section{Momentum sets of bounded unitary representations}\label{Sect1}

The present section records the preliminaries for the results to be obtained later on. 
Throughout this paper, $\cH$ denotes an infinite dimensional complex 
Hilbert space, $B(\cH)$ the $C^*$-algebra of bounded operators on $\cH$,  
and $K(\cH) \trile B(\cH)$ the ideal of compact operators. 
Moreover, we denote by $F(\cH)$ the ideal of finite-rank operators 
and by $\fS_p(\cH)$ the $p$-th Schatten ideal if $1\le p\le\infty$, 
so that $\fS_\infty(\cH)= K(\cH)$.
Accordingly we write $\U_p(\cH) := \U(\cH) \cap (\1 + \fS_p(\cH))$ for the 
corresponding restricted unitary groups. 
We sometimes denote by \wot\ the weak operator topology on $B(\cH)$,  
by \sot\ the strong operator topology, and by $w^*$ the weak-$*$-topology 
of $B(\cH)$ when it is thought of as the topological dual 
of $\fS_1(\cH)$ by using the trace pairing.  
More generally, for any Banach space $\cX$ with the dual $\cX'$ and any subset $A\subeq\cX'$ we denote its norm-closure by $\oline{A}^{\bn}$ 
and its weak-$*$-closure by $\oline{A}^{w^*}$. 

In this note we only consider 
unitary representations $(\pi, \cH)$ of a Banach--Lie group 
$G$ which are {\it bounded} in the sense that 
$\pi \: G \to \U(\cH)$ is a morphism of Banach--Lie groups, 
which is equivalent to $\pi$ being continuous with respect to the 
operator norm on $\U(\cH)$. 
Then the derived representation $\dd \pi \: \g \to \fu(\cH)$ 
is a morphism of Banach--Lie algebras, and we have the {\it momentum 
map} 
$$ \Phi_\pi \: \bP(\cH) = \{ [v] =  \C v \: 0 \not= v \in \cH\} \to \g', 
\quad 
\Phi_\pi([v])(x) := \frac{1}{i}\frac{\la \dd\pi(x)v,v\ra}{\la v,v\ra}, $$
corresponding to the Hamiltonian action of $G$ on the 
K\"ahler manifold $\bP(\cH)$. 

\begin{rem}\label{alternative}
One can provide as follows an alternative description for the momentum map, 
which is specific to the case of the bounded representations. 
For every $v\in\cH\setminus\{0\}$ let us consider the corresponding rank-one orthogonal projection $P_v=\frac{1}{\Vert v\Vert^2}\la\cdot,v\ra v\in B(\cH)$. 
Then $[v]\mapsto P_v$ is a well-defined $\U(\cH)$-equivariant diffeomorphism of $\bP(\cH)$ onto the coadjoint orbit of $\U(\cH)$ consisting of the rank-one orthogonal projections on $\cH$. 
It is easily seen that 
\begin{equation}\label{alternative_eq1}
(\forall [v]\in\bP(\cH))\quad 
\Phi_\pi([v])=-i\trace(\dd\pi(\cdot)P_v)\in\g'.
\end{equation}
If we denote by $\Herm_1(\cH)$ the space of self-adjoint trace-class 
operators on $\cH$, then it follows by the above formula that the momentum map extends 
to a bounded linear map $\Phi_\pi\colon\Herm_1(\cH)\to\g'$, 
$\Phi_\pi(T)=-i\trace(\dd\pi(\cdot)T)$. 
The latter map is the dual $\dd\pi' \: \fu(\cH)' \to \g'$ 
of the derived representation, restricted 
to $\Herm_1(\cH)$, considered as a subspace of $\fu(\cH)'$.
\end{rem}

\begin{defn}\label{mom_def}
The norm-closed convex hull 
$$\Inm_\pi:= \oline{\conv}^{\bn}(\im(\Phi_\pi)) \subeq \g' $$ 
of the image of the momentum map $\Phi_\pi$
is called the {\it norm-closed (convex) momentum set of $\pi$}, while 
the weak-$*$-closed convex hull 
$$ I_\pi := \oline{\conv}^{w^*}(\im(\Phi_\pi)) \subeq \g' $$
is called the {\it (convex) momentum set of $\pi$}. 
It is clear that $\Inm_\pi$ is a weak-$*$-dense subset of 
the weak-$*$-compact set $I_\pi$. 
\end{defn}

For later use we now record a lemma which 
applies in particular to unitary groups of $C^*$-algebras and 
shows the significance 
of the momentum sets in the framework of the contractive representation theory of \cite{Pa83} and \cite{Pa87}. 

\begin{lem}\label{contr}
Let $G$ be a Banach--Lie group with the Lie algebra $\g$. 
Assume that the topologies of $G$ and $\g$ are defined by a left invariant metric $d(\cdot,\cdot)$ and a norm $\Vert\cdot\Vert$, respectively, such that 
$$(\forall x\in\g)\quad 
\limsup_{t\to 0}\frac{d(\exp_G(tx),\1)}{t}\le\Vert x\Vert. $$
Then, for every bounded representation $\pi\colon G\to\U(\cH)$, we have 
$$\sup\{\Vert\mu\Vert\:  \mu\in I_\pi\}
=\Vert\dd\pi\Vert
\le \sup\Bigl\{\frac{\Vert\pi(g_1)-\pi(g_2)\Vert}{d(g_1,g_2)}\:  
g_1,g_2\in G,\ g_1\ne g_2\Bigr\}.$$
\end{lem}

\begin{prf} 
Let us denote the rightmost side of the above quantities by $L_\pi$. 
Since the metric on $G$ is left invariant, it follows that 
$$L_\pi=\sup\Bigl\{\frac{\Vert\pi(g)-\1\Vert}{d(g,\1)}\:  
\1\ne g\in G\Bigr\}.$$
Now let $x\in\g$ arbitrary and define 
$$\gamma\colon\R\to\U(\cH),\quad  
\gamma(t)=\pi(\exp_G(tx))=\exp(t\dd\pi(x)).$$  
Then $\Vert\gamma(t)-\1\Vert\le L_\pi\cdot d(\exp_G(tx),\1)$
and 
$\dd\pi(x)=\lim\limits_{t\to 0}\frac{1}{t}(\gamma(t)-\1)$, 
hence 
$$\Vert\dd\pi(x)\Vert\le L_\pi\limsup_{t\to 0}\frac{d(\exp_G(tx),\1)}{t}\le 
L_\pi\Vert x\Vert.$$
Thus $\Vert\dd\pi\Vert\le L_\pi$. 

Finally, the equality $\sup\{\Vert\mu\Vert\:  \mu\in I_\pi\}
=\Vert\dd\pi\Vert$ is a direct consequence of the fact that  
$\Vert A\Vert=\sup\{\vert\la Av,v\ra\vert/\la v,v\ra\:  0\ne v\in\cH\}$ 
whenever $A=A^*\in B(\cH)$. 
\end{prf}

Below we discuss two key examples of  momentum sets  
(see also Corollary~\ref{use-cor2} below).

\begin{examples}\label{ex_full}
(a) To a trace class operator $T \in \fS_1(\cH)$ we associated 
the continuous linear functional 
$\psi_T(X) := -i\tr(TX)$ on $\fu(\cH)$ and recall that the set of normal 
states of the von Neumann algebra is the set 
\[ \cS := S_*(B(\cH)) 
= \{ i \psi_S \: S^* = S, 0 \leq S, \tr S = 1\}.\] 
The set $S(B(\cH))$ of states of the $C^*$-algebra $B(\cH)$ is much 
larger and contains elements vanishing on the ideal 
$K(\cH)$ of compact operators. 

(b) Let $\id_{\U(\cH)}$ be the identical representation 
of the full unitary group $\U(\cH)$. 
We claim that 
\[  \Inm_{\id_{\U(\cH)}}= -i \cS\res_{\fu(\cH)} 
\quad \mbox{ and } \quad 
I_{\id_{\U(\cH)}}= -i S(B(\cH))\res_{\fu(\cH)}. \] 
Since the convex $\U(\cH)$-invariant set $\cS$ separates of points of $B(\cH)$, it  
is weak-$*$-dense in $S(B(\cH))$ (cf.\ \cite{Se49}; 
\cite[Thm.~X.5.13]{Ne00}). Therefore 
we only have to determine the norm closed momentum set. 

For  $t_1,\dots,t_n\in[0,1]$ with $t_1+\cdots+t_n=1$, unit vectors 
$v_1,\dots,v_n\in\cH$ and $X\in\fu(\cH)$, we have 
$$\sum_{j=1}^n t_j\Phi_{\id_{\U(\cH)}}([v_j])(X)
=-i\sum_{j=1}^n\langle Xv_j,v_j\rangle
=-i \trace(X T)
= \psi_T(X) 
$$
where $T:=\sum\limits_{j=1}^n t_j\langle\cdot,v_j\rangle v_j$ is a 
finite-rank operator in $\cS$. Since the set of these operators is 
norm dense in $\cS$, our claim follows. 

(c) For the identical representation of 
$\U_\infty(\cH)$, we claim that
\[ I_{\id_{U_\infty(\cH)}}^{\bf n}=
 -i \cS \res_{\fu_\infty(\cH)}.\] 
The same argument as in (b) implies that $I_{\id_{U_\infty(\cH)}}^{\bf n} 
= -i \cS \res_{\fu_\infty(\cH)}.$ If $\cH$ is infinite dimensional, 
then this set is not weak-$*$-closed because $0$ is contained in its 
weak-$*$-closure. In fact, for every orthonormal sequence 
$(e_n)_{n \in \N}$ in $\cH$, the sequence of projection 
operators $P_n = \la \cdot, e_n \ra e_n$ in $\Herm_1(\cH) 
\cong \fu_\infty(\cH)'$ converges to $0$ in the weak-$*$-topology 
since $Xe_n \to 0$ holds for every compact operator $X \in K(\cH)$. 
As the set of all non-negative hermitian trace class operators 
with $\|S\|_1 \leq 1$ is weak-$*$-closed, it follows that 
\[ I_{\id_{U_\infty(\cH)}} 
= \{ \psi_S \: S^* = S, 0 \leq S, \tr S = \|S\|_1 \leq 1\}.\] 
\end{examples}

\begin{prop}\label{pullback}
If $\gamma\colon G_1\to G_2$ is any morphism 
of Banach--Lie groups and $(\pi,\cH)$ is a bounded representation of $G_1$, 
then $I_{\pi\circ\gamma}=\{\phi\circ \dd\gamma\:  \phi\in I_\pi\}$. 
\end{prop}

\begin{prf}
The mapping 
$\fg_2'\to \fg_1'$, $\phi\mapsto\phi\circ \dd\pi$ is continuous 
with respect to the weak-$*$-topologies, hence it takes the compact set 
$I_{\pi}$ to a compact set. 
Since the latter compact set is clearly contained in $I_{\pi\circ\gamma}$ 
and contains a weak-$*$-dense subset of $I_{\pi\circ\gamma}$, 
the assertion follows. 
\end{prf} 

\begin{cor}\label{use}
For a bounded representation of the Banach--Lie group $G$,  
the following assertions hold: 
\begin{description}
\item[\rm(i)] 
We have 
$I_\pi=\{-i\psi\circ \dd\pi\:  \psi\in S(B(\cH))\}$.
\item[\rm(ii)] 
If $\pi(G)\subseteq U_\infty(\cH)$, 
then $I_\pi=\{-i\psi\circ \dd\pi\:  \psi\in S_*(B(\cH))\}$.  
\end{description}
\end{cor}

\begin{prf}
Use Proposition~\ref{pullback} along with Examples~\ref{ex_full}. 
\end{prf}

The next corollary shows that the momentum set of a representation 
of a unital $C^*$-algebra (\cite[Def.~X.5.9 and Thm.~X.5.13(iii)]{Ne00}) 
coincides with the momentum set of the corresponding 
representation of its unitary group.

\begin{cor}\label{use-cor2}
Consider a unital $C^*$-algebra $\cA\subseteq B(\cH)$ 
with the tautological representation of its unitary group 
denoted by $\id_{\U(\cA)}\colon \U(\cA)\to B(\cH)$. 
The momentum set of this representation is 
$I_{\id_{\U(\cA)}}=\{-i\psi|_{\fu(\cA)}\:  \psi\in S(\cA)\}$. 
\end{cor}

\begin{prf}
The assertion follows by Corollary~\ref{use} 
along with the fact that every state of $\cA$ 
extends to some state of $B(\cH)$ (cf.~\cite{Dix64}). 
\end{prf}

We now state a consequence of Corollary~\ref{use} on extreme points of momentum sets for a special class of bounded representations. 

\begin{cor}\mlabel{use-cor1}
If $G$ is a Banach--Lie group and $\pi\colon G\to \U_\infty(\cH)$ 
a bounded representation,   
then for every extreme point $\lambda\in I_\pi$ 
there exists a unit vector $v\in\cH_\pi$ such that 
$\lambda(X)=-i\la \dd\pi(X)v,v\ra$ for every $X\in\fg$, i.e., 
\[ \Ext(I_\pi) \subeq \im(\Phi_\pi). \]
\end{cor}

\begin{prf} 
It follows by Corollary~\ref{use}(ii) that $I_\pi$ 
is the image of the weak-$*$-compact convex set $S_*(B(\cH))$ 
by the affine transform $\psi\mapsto-i\psi\circ \dd\pi$. 
Then every extreme point of $I_\pi$ is the image of some extreme point of~$S_*(B(\cH))$. 
Since every normal pure state of $B(\cH)$ is of the form $T\mapsto \la Tv,v\ra$ 
for some unit vector $v\in\cH_\pi$, the assertion follows. 
\end{prf}

\subsection*{Dual Banach--Lie groups and normal representations}

\begin{defn} (a) By 
\emph{dual Banach--Lie algebra} we mean any pair $(\g,\g_*)$ 
consisting of a Banach--Lie algebra $\g$ and a closed 
linear subspace $\g_*$ of the topological dual $\g'$ such that 
for every continuous linear functional $\psi\colon\g_*\to\R$ 
there exists a unique element $x_\psi\in\g$ such that 
$\psi(\xi)=\xi(x_\psi)$ for every $\xi\in\g'$. 
If this is the case, then we have a linear topological isomorphism 
$(\g_*)'\simeq\g$, $\psi\mapsto x_\psi$. 
Therefore $\g_*$ is called a \emph{predual} of $\g$ and it makes sense 
to speak about the weak-$*$-topology of $\g$. 
The predual will be fixed in what follows, and it will be omitted 
from the notation, for the sake of simplicity. 

(b) If $\g$ is the Lie algebra of some Banach--Lie group $G$, 
then $G$ is said to be a {\it dual Banach--Lie group}. 
By a {\it normal representation} of $G$ we mean 
any bounded representation $\pi\colon G\to \U(\cH)$ 
whose derived representation $\dd \pi \colon \g \to \fu(\cH)$ 
has the property that its adjoint 
$\dd\pi' \: \fu(\cH)' \to \g'$ maps $\fu(\cH)_* := i\Herm_1(\cH)$ 
into $\g_*$. 
This is equivalent to $\dd\pi$ being 
continuous with respect to the weak-$*$-topologies on $\g$ and $\fu(\cH)$. 
\end{defn}

\begin{ex} For every $W^*$-algebra $\cM$ with predual $\cM_*$, 
the unitary group $\U(\cM)$ is a dual Banach--Lie group 
with respect to the predual 
\[ \fu(\cM)_* := \{ \phi \in \cM_* \: \phi^* = - \phi\}. \] 
This applies in particular to 
$\cM = B(\cH)$ with $\cM_* = \fS_1(\cH)$. 
The restriction of a normal representation 
$\pi \: \cM \to B(\cH)$ to the unitary group 
defines a normal representations of the dual Banach--Lie group 
$\U(\cM)$.   
\end{ex}

\begin{lemma}\label{dual}
For a dual Banach--Lie group $G$, the following assertions hold: 
\begin{description} 
\item[\rm(i)] 
For every normal representation $\pi\colon G\to \U(\cH)$ we have 
$\Inm_\pi\subseteq\g_*$. 
\item[\rm(ii)] \ \ 
Let $\fa$ be closed subalgebra of $\g$ 
and $A=\la\exp_G(\fa)\ra\subseteq G$ be the corresponding integral 
subgroup, endowed with its canonical Lie group structure. 
Denote by $\fa^*$ the linear subspace of the topological dual $\fa'$ 
consisting of the weak-$*$-continuous functionals. 
If the topology of $\g$ is defined by a norm  
such that the unit ball of $\fa$ is weak-$*$-dense in 
the closed unit ball of $\g$,  
then the restriction mapping $R_{\fa}\colon\g_*\to\fa^*$, $\xi\mapsto\xi\vert_{\fa}$, 
is an isometric linear isomorphism and for every normal representation $\pi\colon G\to \U(\cH)$ we have $R_{\fa}(\Inm_\pi)=\Inm_{\pi\vert_A}\subeq\fa^*$. 
\end{description}
\end{lemma}

\begin{prf} (i) For every $v\in\cH\setminus\{0\}$ let us consider 
the corresponding rank-one orthogonal projection 
$P_v=\frac{1}{\Vert v\Vert^2}\la\cdot,v\ra v\in B(\cH)$. 
Then Remark~\ref{alternative} shows that $\Phi_\pi([v])=-i\trace(\dd\pi(\cdot)P_v)$. 
Since $\pi\colon G\to B(\cH)$ is a normal representation, 
$\Phi_\pi([v])\in\g_*$. 
The predual $\g_*$ is a norm-closed linear subspace of $\g'$, 
hence the norm-closed convex hull $\Inm_\pi$ of the image of $\Phi_\pi$ 
is contained in~$\g_*$. 

(ii) The hypothesis that the unit ball of $\fa$ is weak-$*$-dense in 
the closed unit ball of $\g$ entails that for every 
$\xi\in\g_*$ we have $\Vert\xi\Vert=\Vert\xi\vert_{\fa}\Vert$. 
Thus the restriction mapping $R_{\fa}\colon\g_*\to\fa^*$ is an isometry, and in particular its 
range is a closed subspace of $\fa'$. 
Moreover, since, every weak-$*$-continuous linear functional on $\fa$ 
extends to a weak-$*$-continuous linear functional on $\g$ by the Hahn-Banach theorem, the restriction map $R_{\fa}\colon\g_*\to\fa^*$ is onto.

On the other hand, it follows directly from the definition of the momentum maps 
that for any normal representation $\pi\colon G\to \U(\cH)$ we have  
$\Phi_{\pi\vert_A}=R_{\fa}\circ\Phi_\pi$.  
The fact that $R_{\fa}\colon\g_*\to\fa^*$ is an isometric linear isomorphism implies 
that it commutes with the operation of taking the norm-closed convex hull. 
This proves~(ii). 
\end{prf}

\begin{prop}\label{dual_appl}
Let $\cM$ be a $W^*$-algebra with a weak-$*$-dense unital $C^*$-subalgebra $\cA\subeq\cM$. 
Then the following assertions hold: 
\begin{description} 
\item[\rm(i)]
If $\pi\colon \U(\cM)\to \U(\cH)$ is a normal representation, then 
the restriction mapping 
$R_{\fu(\cA)}\colon\fu(\cM)_*\to\fu(\cA)^*$  
is an isometric linear isomorphism 
satisfying $R_{\fu(\cA)}(\Inm_{\pi})=\Inm_{\pi\vert_{\U(\cA)}}$. 
\item[\rm(ii)] 
If $\pi_j\colon \U(\cM)\to \U(\cH_j)$, 
$j =1,2$, are normal representations such that $I_{\pi_1\vert_{\U(\cA_0)}}\ne I_{\pi_2\vert_{\U(\cA_0)}}$ for some weak-$*$-dense unital $C^*$-subalgebra $\cA_0\subeq\cM$, 
then  
$\Inm_{\pi_1\vert_{\U(\cA)}}\ne \Inm_{\pi_2\vert_{\U(\cA)}}$ as well. 
\end{description}
\end{prop}

\begin{prf} 
(i) According to Kaplansky's Density Theorem  
for the weak-$*$-dense $C^*$-subalgebra $\cA\subeq\cM$,  
it follows that the unit ball of $\cA$ is weak-$*$-dense in the closed unit ball of $\cM$ 
(see for instance Thm.~4.8 and the comment following 
its proof in \cite[Ch.~2]{Ta02}). 
It is then clear that a similar assertion holds for the self-adjoint parts of the unit balls.  
Now the assertion follows by Lemma~\ref{dual}(ii). 

(ii) Since $\Inm_{\pi_j\vert_{\U(\cA_0)}}$ is 
weak-$*$-dense in the momentum set $I_{\pi_j\vert_{\U(\cA_0)}}$ for $j=1,2$ 
and $I_{\pi_1\vert_{\U(\cA_0)}}\ne I_{\pi_2\vert_{\U(\cA_0)}}$, 
it follows that $\Inm_{\pi_1\vert_{\U(\cA_0)}}\ne\Inm_{\pi_2\vert_{\U(\cA_0)}}$.

On the other hand, by using (i) for both $\cA$ and $\cA_0$,  
we get the isometric linear isomorphism 
$$R:=R_{\fu(\cA)}\circ(R_{\fu(\cA_0)})^{-1}\colon \fu(\cA_0)^*\to\fu(\cM)_*\to\fu(\cA)^*$$ 
with $R(\Inm_{\pi_j\vert_{\U(\cA_0)}})=\Inm_{\pi_j\vert_{\U(\cA)}}$ for $j=1,2$. 
Since we have already seen that $\Inm_{\pi_1\vert_{\U(\cA_0)}}\ne\Inm_{\pi_2\vert_{\U(\cA_0)}}$, 
it then follows that 
$\Inm_{\pi_1\vert_{\U(\cA)}}\ne \Inm_{\pi_2\vert_{\U(\cA)}}$.
\end{prf}

\section{Irreducible unitary representations of $\U_\infty(\cH)$}\label{Sect2}

In this section we provide a brief review of the irreducible representation theory of the unitary group $\U_\infty(\cH)$, 
thereby updating Kirillov's classification in \cite{Ki73} by using the Lie theoretic tools available nowadays in infinite dimensions: 
the momentum sets, the highest weight representations, and their geometric realizations in line bundles over flag manifolds. 
In particular, we show in Proposition~\ref{prop:mostbasic} that two of the aforementioned irreducible representations are unitarily equivalent if and only if their momentum sets coincide. 
The discussion of the present section prepares the ground for more general results in representation theory of unitary groups of $C^*$-algebras, 
which will be obtained in the next section. 

\subsection*{Highest weight theory}

Bounded unitary representations of $\U_\infty(\cH)$, or, 
equivalently, corresponding 
holomorphic representations of the complexified group 
\[ \GL_\infty(\cH) = \GL(\cH) \cap (\1 + K(\cH)),\]   
have 
been classified for general Hilbert spaces $\cH$ in \cite{Ne98}, 
where it is also shown that they all decompose as direct sums of 
irreducible ones and that the irreducible ones are 
highest weight representations $(\pi_\lambda, \cH_\lambda)$. 
To make this more precise, we choose an orthonormal basis 
$(e_j)_{j \in J}$ in $\cH$, so that $\cH \cong \ell^2(J,\C)$. 
Then the Lie subalgebra $\ft$ of diagonal operators in $\fu_\infty(\cH)$ 
is maximal abelian and, as a Banach space, isomorphic 
to $c_0(J,\R)$. The corresponding subgroup 
$T = \exp(\ft) \cong c_0(J,\R)/c_0(J,\Z) \cong c_0(J,\R)/\Z^{(J)}$ 
is the analog of a maximal torus in $\U_\infty(\cH)$. 

\begin{defn}\label{def_weights} 
(a) The {\it Weyl group} in our setting is  the group of finite permutations of $J$, 
and we denote it by $\cW = S_{(J)}$. It acts naturally by 
composition on the set $\Z^J$ of $\Z$-valued functions on $J$.

The {\it group of weights} is $\cP = \ell^1(J,\Z) \cong \Z^{(J)}$, 
that is, the additive group of all finitely supported functions $\lambda \: J \to \Z$, $j \mapsto \lambda_j $. 
It can also be identified with the character group of the Banach--Lie group $T = \exp \ft$ by assigning to $\lambda$ the character given by 
$\chi_\lambda(t) := \prod_{j \in J} t_j^{\lambda_j}$. 
We write $\eps_j \in \cP$ for the function 
defined by $\eps_j(k) = \delta_{jk}$. 

For every $\lambda\in\cP$ we write 
$\lambda_\pm := \max(\pm\lambda, 0)$, and we thus obtain a 
decomposition $\lambda = \lambda_+ - \lambda_-$. 
For every $\lambda\in\cP$ we denote by $[\lambda]=\cW\lambda$ the 
orbit of $\lambda$ with respect to the natural action of $\cW$ on $\cP$, 
and write $\cP/\cW:=\{[\lambda]\: \lambda\in\cP\}$ for the set of 
$\cW$-orbits in $\cP$. 

(b) Each $\lambda\in\cP$ defines a continuous linear functional in $\ft' \cong \ell^1(J,\R)$, 
and there exists a unique unitary representation $(\pi_\lambda, \cH_\lambda)$ 
of $\U_\infty(\cH)$ whose weight set is given by 
\begin{equation}
  \label{eq:weightset}
\cP_\lambda = \conv(\cW \lambda) \cap \cP, 
\end{equation}
(see  \cite{Ne98}). Here the uniqueness implies in particular
that $\pi_\mu \cong \pi_\lambda$ for $\mu \in \cW\lambda$, 
so that the equivalence classes of these representations are 
parametrized by the orbit space $\cP/\cW$. 
\end{defn}

Let 
$$ \Part(n,k) := 
\Big\{ \lambda = (\lambda_1,\ldots, \lambda_k) \in \N^k \:  
\lambda_1 \geq \ldots \geq \lambda_k > 0, \sum_{j=1}^k \lambda_j = n \Big\} $$
be the set of partitions of $n \in \N_0$ into $k$ pieces and 
put 
$$ \Part(n) := \bigcup_{k \leq n} \Part(n,k). $$

The following theorem explains how pairs of partitions parametrize the 
equivalence classes of bounded unitary representations of 
$\U_\infty(\cH)$. Below we shall briefly discuss how this fits 
Kirillov's classification of the continuous unitary representations 
of $\U_\infty(\cH)$ for the case where $\cH$ is separable 
(\cite{Ki73}).

\begin{thm}\label{thm:parti}
If $\cH$ is an arbitrary infinite dimensional complex Hilbert space, then the following assertions hold
for the representations $(\pi_\lambda)_{\lambda \in \cP}$ of $\U_\infty(\cH)$.  
\begin{description}
\item[\rm(i)] For $\lambda,\mu\in\cP$ we have 
$ \pi_\lambda \cong \pi_\mu$ if and only if 
$\mu \in \cW\lambda$. 
\item[\rm(ii)] There exists a bijective map 
$$\cP/\cW \to \bigcup_{n,m \in \N} \Part(n) \times \Part(m), 
\quad [\lambda]\mapsto((m^{\lambda_+}_k)_{k\ge1},(m^{\lambda_-}_k)_{k\ge1})$$
where $ m^{\lambda}_k := \vert\{ j \in J \: \lambda_j = k\}\vert$ whenever 
$\lambda\in\cP$ and $k \in \Z \setminus \{0\}$. 
\item[\rm(iii)] For every $\lambda\in\cP$ we have 
$\pi_\lambda\cong\pi_{\lambda_+}\otimes\pi_{\lambda_-}^*$. 
\end{description}
\end{thm}

\begin{prf} (i) From 
\eqref{eq:weightset} and the fact that 
\begin{equation} \label{eq:weightext}
\cW\lambda = \Ext(\conv(\cW\lambda)) = \Ext(\conv(\cP_\lambda)) 
\end{equation}
is the set of extreme points of $\conv(\cP_\lambda)$
(\cite[proof of Thm.~1.20]{Ne98}), 
it follows that the equivalence class of the representation 
determines $\cW\lambda$ because it determines its weight set. 

(ii) It is easy to see that the 
Weyl group orbit of any $\lambda\in\cP$ 
is uniquely determined by the numbers 
$m^\lambda_k $ with $ k \in \Z \setminus \{0\}$. 
The property 
$\mu \in \cW\lambda$ is equivalent to 
$\mu_\pm \in \cW \lambda_\pm$. 
For $\lambda \geq 0$, the number 
$n_\lambda := \sum_{j \in J}\lambda_j$ is an invariant of the 
Weyl group orbit and $\lambda$ describes a partition 
of $n_\lambda$ into $m_\lambda := \sum_k m_k$ summands. 
We thus get a one-to-one correspondence 
between partitions of $n \in \N$ and the set of Weyl group 
orbits in the set $\{ \lambda \in \Z^{(J)} \:  \lambda\geq 0, 
n_\lambda = n \}$. This leads to the asserted bijection. 

(iii) This was already noted in \cite[Subsect.~2.14]{Ol90} for separable Hilbert spaces; see also \cite[Prop.~IX.1.15]{Ne00} for an algebraic version. 

The present assertion can be obtained as follows.  
It is straightforward to check that the unitary representations $\pi_\lambda$ and $\pi_{\lambda_+}\otimes\pi_{\lambda_-}^*$ have the same sets of weights 
and they are highest weight representations as in \cite[Def.~III.6(c)]{Ne98}.  
Moreover, $\pi_\lambda$ is an irreducible representation.  
Therefore, due to the fact that the irreducible highest weight representations are uniquely determined by their sets of weights  
(see \cite[Cor.~I.15]{Ne98}), the assertion will follow as soon as we have proved that 
the unitary representation $\pi_{\lambda_+}\otimes\pi_{\lambda_-}^*$ is irreducible as well. 

To this end, denote by $\Sigma$ the family of all countable subsets of $S\subeq J$ with $\supp\lambda\subset S$. 
For every $S\in\Sigma$ denote $\cH_S=\ell^2(S,\C)\hookrightarrow \ell^2(J,\C)=\cH$ and  
use the orthogonal decomposition 
$\cH=\cH_S\oplus\cH_S^\perp$ to construct the natural embedding 
\[
\GL_\infty(\cH_S)\hookrightarrow\GL_\infty(\cH),\quad  
g\mapsto\begin{pmatrix}g & 0\\0 &\1\end{pmatrix}.
\]
It then follows that the directed union of subgroups 
$\bigcup_{S\in\Sigma}\GL_\infty(\cH_S)$ is dense in the Banach--Lie group  
$\GL_\infty(\cH)$, since the closure of the range of any compact operator is a separable subspace.  
If $v_\pm\in\cH_{\lambda_\pm}$ is a primitive vector for the highest weight representation $\pi_{\lambda_\pm}$, 
then \cite[Prop.~IX.1.14]{Ne00} shows that 
$\pi_{\lambda_\pm,S}\colon \GL_\infty(\cH_S)\to \GL(\cH_{\lambda_\pm,S})$ 
is an  irreducible highest weight representation of $\GL_\infty(\cH_S)$ with highest weight $\lambda_\pm$, where $\cH_{\lambda_\pm,S}$ 
stands for the closed linear subspace of $\cH_{\lambda_\pm}$ generated by 
$\pi_{\lambda_\pm}(\GL_\infty(\cH_S))v_{\lambda_\pm}$. 
Since $\cH_S$ is a separable Hilbert space, it follows by the aforementioned remark of \cite[Subsect.~2.14]{Ol90} that the representation 
$\pi_{\lambda_+,S}\otimes\pi_{\lambda_-,S}^*$ of $\GL_\infty(\cH_S)$ is irreducible for each $S\in\Sigma$. 
Then $\pi_{\lambda_+}\otimes\pi_{\lambda_-}^*$ is an irreducible representation of $\GL_\infty(\cH)$ by Proposition~\ref{inductive}, and this completes the proof, as discussed above. 
\end{prf}

For any $\lambda\in\cP$ let  $I_\lambda \subeq \fu_\infty(\cH)' \cong \Herm_1(\cH)$ 
denote the momentum set of $\pi_\lambda$. Then 
$I_\lambda$ is a weak-$*$-closed bounded convex subset, hence 
weak-$*$-compact. This implies that the projection map 
$$ p_\ft \: \fu_\infty(\cH)' \to \ft' \cong \ell^1(J,\R), $$
maps $I_\lambda$ onto a weak-$*$-closed compact subset of 
$\ft$, which therefore coincides with the momentum set 
of the representation $\pi_\lambda\res_T$ of the diagonal group 
(Proposition~\ref{pullback}).  
Since this representation decomposes into weight spaces, 
\eqref{eq:weightset} leads to 
\begin{equation}
  \label{eq:projmom}
p_\ft(I_\lambda) = I_{\pi_\lambda}\res_\ft
= \oline{\conv}^{w^*}(\cW\lambda) =: \co(\lambda). 
\end{equation}

\begin{lem}\label{lem:1.2} 
If $\lambda, \mu \: J \to \R$ are finitely supported 
with $\mu\not\in \cW\lambda$, then we have 
$\co(\lambda) \not= \co(\mu)$ for the weak-$*$-closed convex 
hulls of the Weyl group orbits in $\ft' \cong \ell^1(J,\R)$. 
\end{lem}

\begin{prf} First we observe that 
$\lambda^{-1}(0) \cap \mu^{-1}(0)$ is cofinite in $J$, 
which implies that there exists a partial order 
$\leq$ on $J$ for which 
$\lambda \: J \to \Z$ is non-increasing and for 
$\mu$ the set 
$\{ (i,j) \: i < j, \mu_i < \mu_j\}$ 
of inversions is finite, so that 
$\cW\mu$ contains a non-increasing element, and we may w.l.o.g.\ 
assume that $\mu$ is also non-increasing. Then 
$$ J = J_+ \dot\cup J_0 \dot\cup J_-, \quad 
J_+ < J_0 < J_-, $$
where $\lambda$ and $\mu$ vanish on $J_0$, are 
$\geq 0$ on the finite set $J_+$ and 
$\leq 0$ on the finite set $J_-$. 

As we have seen in the proof of Theorem~\ref{thm:parti}, 
$\mu \not\in \cW\lambda$ implies that either 
$\mu_+ \not\in \cW\lambda_+$ or 
$\mu_- \not\in \cW\lambda_-$. We assume without loss of generality 
that $\mu_+ \not\in \cW\lambda_+$. 
Since $\cW\lambda_+ = \Ext(\conv(\cW\lambda_+))$ 
by \eqref{eq:weightext}, 
we then either have $\mu_+ \not\in \conv(\cW\lambda_+)$ 
or vice versa. We assume that $\mu_+ \not\in \conv(\cW\lambda_+)$. 
Then the Schur Convexity Theorem implies that there 
exists a minimal $j_0 \in J_+$ with 
$$ \sum_{j \leq j_0} \lambda(j) < \sum_{j \leq j_0} \mu(j). $$
Put $x_j = 1$ for $j \leq j_0$ and $x_j =0$ otherwise 
and note that $x$ has finite support and satisfies 
$\lambda(x) < \mu(x).$
For any other element $\lambda' \in \cW\lambda$, 
the difference $\lambda- \lambda'$ is a sum of positive 
roots $\eps_i - \eps_j$, $i < j$, so that 
$x_i \geq x_j$ for $i < j$, implies that 
$$ \mu(x) > \lambda(x) \geq \lambda'(x) $$ 
(see Lemma~\ref{lemma_app2}). 
We conclude that 
$\mu$ is not contained in the weak-$*$-closed convex hull 
of $\conv(\cW\lambda)$ in $\ell^1(J,\R)$. 
The other cases are treated similarly. 
\end{prf}

\begin{prop}\mlabel{prop:mostbasic} 
If $\lambda, \mu \in \cP$ and $\pi_\lambda\not\cong \pi_\mu$, then the 
corresponding momentum sets $I_\lambda$ and $I_\mu$ in 
$\fu_\infty(\cH)'$ are different.   
\end{prop}

\begin{prf} If $\pi_\lambda$ and $\pi_\mu$ are not equivalent, then 
$\mu \not\in \cW\lambda$ (Theorem~\ref{thm:parti}), 
so that the preceding proposition 
implies that the weak-$*$-closed convex hulls of 
$\cW\lambda$ and $\cW\mu$ in $\ell^1(J,\R)$ are different. 
In view of \eqref{eq:projmom}, this leads to 
$p_\ft(I_\lambda) \not= p_\ft(I_\mu)$, which implies in 
particular that $I_\lambda \not= I_\mu$. 
\end{prf} 

\subsection*{Realization in tensor products} 

To match the highest weight approach from \cite{Ne98} with 
Kirillov's classification in terms of Schur--Weyl theory (\cite{Ki73}), 
it suffices  to explain how 
the representation $\pi_\lambda$ with highest weight 
$\lambda \geq 0$ occurs in the decomposition of 
some tensor product $\cH^{\otimes n}$ (cf.\ Theorem~\ref{thm:parti}(iii)). 

Let $n := n_\lambda = \sum_{j \in J}\lambda_j$ and suppose that 
$\lambda \in \Part(n,k)$. Then $\lambda$ defines a 
Ferrers diagram (shape), containing 
$\lambda_j$ boxes in the $j$th row (cf.\ \cite[Sect.~8.1.2]{GW98}). 
A tableau of shape $\lambda$ is a bijective assignment of the integers 
$1,\ldots, n$ to the boxes of~$\lambda$. Clearly, 
the group $S_n$ acts simply transitively on the set 
$\Tab(\lambda)$ of all tableaux of shape $\lambda$. 

We may w.l.o.g.\ assume that $\N \subeq J$, so that we have 
orthonormal vectors $(e_j)_{j \in \N}$ in $\cH$. 
For $T \in \Tab(\lambda)$ we set $i_j = r$ if $j$ occurs in the 
$r$th row of $T$ and define 
$$ e_T := e_{i_1} \otimes \cdots \otimes e_{i_n} \in \cH^{\otimes n}. $$
The Young symmetrizer $s(T) \in \C[S_n]$ is the product 
$s(T) = c(T)r(T)$, where 
$$ r(T) = \sum_{ \sigma \in R_T} \sigma \quad \mbox{ and } \quad 
c(T) = \sum_{ \sigma \in C_T} \sgn(\sigma)\sigma, $$
and $C_T\cong \prod S_{\lambda_i} \subeq S_n$ is the subgroup preserving 
the column partition of $\{1,\ldots, n\}$ 
and $R_T$ is the subgroup preserving the 
row partition. Then 
$c(T)$ projects $\cH^{\otimes n}$ onto the subspace 
$$ \Lambda^{\lambda_1}(\cH) \otimes \cdots \otimes \Lambda^{\lambda_k}(\cH)$$ 
and from \cite[Lemma~9.3.1]{GW98} we derive that $0\not= c(T)e_T$ 
is a vector of weight $\lambda$, fixed by the group $N$ of 
upper triangular matrices in $\GL(\cH)$ with respect to any 
linear order on $J$ for which $1,\ldots, k$ come first 
(so that  $\lambda = (\lambda_1, \ldots, \lambda_k)$ defines a 
dominant weight). Moreover, 
$$ M^\lambda := \sum_{T \in \Tab(\lambda)} \C c(T) e_T $$ 
is an $S_n$-submodule of $\cH^{\otimes n}$ isomorphic to the 
multiplicity space of the highest weight representation 
$\tilde\pi_\lambda^\C$ of $\GL(\cH)$ in $\cH^{\otimes n}$ 
(cf.\ \cite[Prop.~9.3.4, Thm.~9.3.9]{GW98} and Appendix~\ref{app:a}).

\subsection*{Flag manifolds}

Let $(e_j)_{j\in J}$ be an orthonormal basis in $\cH$ 
and $\lambda \: J \to \Z$ be finitely supported, 
so that 
\begin{equation*}
D_\lambda :=\sum\limits_{j\in J}\lambda_j\la\cdot,e_j\ra e_j 
\end{equation*}
is a hermitian finite rank operator, hence in particular 
of trace class, so that we get a continuous linear functional 
\begin{equation}\label{psi}
\psi_\lambda := \psi_{D_\lambda} 
\colon B(\cH)\to\C,\quad \psi_\lambda(T)=-i\trace(D_\lambda T).
\end{equation}

If $\cH_n := \ker(D_\lambda - n \1)$, $n \in \Z$, denotes the eigenspaces 
of $D_\lambda$, then the assumptions on $\lambda$ imply that 
only finitely many are non-zero and only $\cH_0 = \ker D_\lambda$ is 
infinite dimensional. 

We write 
\[ \cN_\lambda := \Big\{ X \in B(\cH) \: (\forall n \in \Z)\, 
X \cH_n \subeq \sum_{m < n} \cH_m\Big\} \] 
for the subalgebra of $B(\cH)$ 
consisting of the operators which are block strictly 
upper-triangular with respect to the spectral scale of 
the self-adjoint operator $D_\lambda$ 
and note that 
\[ \{D_\lambda\}' = \{ X \in B(\cH) \: (\forall n \in \Z)\, 
X \cH_n \subeq \cH_n\} \] 
is the commutant of $\{D_\lambda\}$. 
It is easily seen that we have a triple decomposition for $B(\cH)$ 
given by  the direct sum decomposition 
\begin{equation}\label{compl}
\cN_\lambda\oplus\{D_\lambda\}'\oplus\cN_\lambda^*=B(\cH).
\end{equation}

Let $N := \|\lambda\|_\infty$. 
Then the eigenspaces $\cH_n$ of $D_\lambda$ define the finite flag 
\[ \cF = (F_{-N}, \ldots, F_N), \quad 
F_k := \sum_{n \leq k} \cH_k \] 
of closed subspace of $\cH$. 
The space $\Gr_\cF := \GL(\cH)\cF \cong \GL(\cH)/P_\cF$ 
is called the corresponding {\it flag manifold}. Here 
\[ P_\cF := (\cN_\lambda \oplus \{D_\lambda\}')^\times 
= \{ g \in \GL(\cH) \: (\forall n) g F_n = F_n \} \] 
is the stabilizer of $\cF$. Since its Lie algebra 
$\fp_\cF = \cN_\lambda + \{D_\lambda\}'$ is complemented by 
the closed subspace $\cN_\lambda^*$, $\Gr_\cF$ carries the 
structure of a complex Banach manifold for which the quotient 
map $\GL(\cH) \to \Gr_\cF, g \mapsto g\cF$ is a holomorphic submersion. 

It is easy to see that a flag 
$\cF' = (F_{-N}', \ldots, F_N')$ of closed subspaces of $\cH$ 
is contained in $\Gr_\cF$ if and only if 
\[ \dim F_n = \dim F_n' \quad \mbox{ for } \quad n < 0 
\quad \mbox{ and } \quad 
\dim F_n^\bot = \dim F_n'^\bot \quad \mbox{ for } \quad n \geq 0\] 
and hence that the subgroup $\U_\infty(\cH) \subeq \U(\cH)$ 
act transitively on $\Gr_\cF$. 

The stabilizer of $\cF$ in $\U(\cH)$ is 
\begin{equation}
  \label{eq:stabuni}
\U(\cH)_\cF = \U(\cH)_\lambda := \U(\cH) \cap \{D_\lambda\}' = 
\{ u \in \U(\cH)\: (\forall n)\, u\cH_n = \cH_n\}, 
\end{equation}
so that $\Gr_\cF \cong \U(\cH)/\U(\cH)_\lambda$ is a Banach 
homogeneous space of $\U(\cH)$ (and similary of $\U_\infty(\cH)$).

Now let $\widetilde{\pi}_\lambda^{\C}\colon\GL(\cH)\to B(\cH_\lambda)$ 
be the canonical extension of the highest weight representation 
corresponding to $\lambda$ as in \cite[Cor.~III.11 and Thm.~III.15]{Ne98} 
and pick a unit vector $v_\lambda\in\cH_\lambda$ of weight $\lambda$.

In the following theorem, we shall use the realization of a 
Hilbert space by holomorphic sections of a line bundle: 

\begin{rem} \mlabel{rem:5.9} If $\cH$ is a complex Hilbert space, 
then its projective space $\bP(\cH)$ is a complex Hilbert manifold. 
Moreover, there exists a holomorphic line bundle 
$q \: \bL_\cH \to \bP(\cH)$ with the property that for 
every non-zero continuous linear functional 
$\alpha \in \cH^*$ we have on the open subset
$U_\alpha := \{ [v] \in \bP(\cH) \: \alpha(v) \not=0\}$ 
a bundle chart 
$$ \phi_\alpha \: (\bL_\cH)\res_{U_\alpha} \to U_\alpha \times \C $$
such that the transition functions are given by 
$$ \phi_\beta \circ \phi_\alpha^{-1}([v],z) 
= \Big([v], \frac{\alpha(v)}{\beta(v)}\Big)\quad \mbox{ for } \quad 
0\not=\alpha, \beta \in \cH^*. $$ 
This implies that each $0\not=v \in \cH$ defines a linear functional 
on the fiber $(\bL_\cH)_{[v]}$ by 
$$ \phi_\alpha^{-1}([v], z) \mapsto \alpha(v) z, $$
which further implies that $(\bL_\cH)_{[v]}^* = [v]$, 
i.e., $\bL_\cH^*$ is the tautological bundle over $\bP(\cH)$. 

The complement of the zero-section 
of $\bL_\cH$ is equivalent, as a $\C^\times$-bundle,  
to the projection $\cH\setminus \{0\} \to \bP(\cH)$ by 
the map $\phi_\alpha([v],z) \mapsto \frac{1}{z \alpha(v)}v$. 
This identification can be used to show that the natural map 
$$ \Psi \: \cH^* \to \Gamma(\bL_\cH), \quad 
\Psi(\alpha)([v]) = \phi_\beta^{-1}\Big([v], \frac{\alpha(v)}{\beta(v)}\Big) 
\quad \mbox{for } \quad 
\beta(v)\not=0 $$
defines a linear isomorphism (see \cite[Thm.~V.4]{Ne01} for details). 

As the group $\U(\cH)$ acts smoothly by holomorphic bundle 
isomorphisms on $\bL_\cH$, this construction shows that the 
unitary representation $\pi^* \: \U(\cH) \to \U(\cH^*)$, 
given by $\pi^*(g)\alpha = \alpha \circ \pi(g)^*$ 
can be realized in the space $\Gamma(\bL_\cH)$ of holomorphic 
sections of $\bL_\cH$. 

To realize the identical representation on $\cH$ itself 
by holomorphic sections, we simply exchange the role 
of $\cH$ and $\cH^*$, which leads to a $\U(\cH)$-equivariant isomorphism 
$\cH \to \Gamma(\bL_{\cH^*})$. 
\end{rem}

We now establish a partial generalization of \cite[Prop.~7.2]{Ne10}.

\begin{thm}\label{flags} 
The following assertions hold: 
\begin{description}
 \item[\rm(i)]
For every $X\in \fp_\cF^* = \cN_\lambda^*\oplus\{D_\lambda\}'$ we have 
\begin{equation}
  \label{eq:eigenvect}
\dd\widetilde{\pi}_\lambda^{\C}(X)v_\lambda= i\psi_\lambda(X)v_\lambda
\quad \mbox{ with } \quad 
\psi_\lambda(X) = 0 \quad \mbox{ for } \quad 
X \in \cN_\lambda^*,\end{equation}
and $\la \dd\tilde\pi_\lambda^\C(X)v_\lambda, v_\lambda \ra = 0$  
for $X \in \cN_\lambda$. 
\item[\rm(ii)] For $\alpha_\lambda(w) := \la w, v_\lambda\ra$, 
we obtain a holomorphic equivariant map 
\[ \eta \: \U(\cH)/\U(\cH)_\lambda
\cong \GL(\cH)/P_\lambda \cong \Gr_\cF \to \bP(\cH_\lambda^*), \quad 
u\cF \mapsto [\tilde\pi_\lambda^*(u)\alpha_\lambda]\] 
and if $\bL_\lambda := \eta^*\cL_{\cH_\lambda^*}$ is the pullback of the 
canonical line bundle on $\bP(\cH_\lambda^*)$ with 
$\Gamma_{\rm hol}(\cL_{\cH_\lambda^*}) \cong \cH_\lambda$, then 
$\cL_\lambda$ is $\U(\cH)$-equivariant and we obtain a realization 
of $\cH_\lambda$ as a Hilbert space of holomorphic sections of 
$\bL_\lambda$. 
\item[\rm(iii)] The coadjoint $\U(\cH)$-orbit $\cO_{\psi_\lambda}$ of 
$\psi_\lambda\vert_{\fu(\cH)}$ is a homogeneous Banach manifold 
isomorphic to $\Gr_\cF$ with a $\U(\cH)$-invariant strong K\"ahler structure, 
and we have 
$$\Phi_{\widetilde{\pi}_\lambda}([v_\lambda])=\psi_\lambda\vert_{\fu(\cH)} 
\quad\text{ and }\quad
I_{\widetilde{\pi}_\lambda}=\oline{\conv}^{w^*}\big(\cO_{\psi_\lambda}\big).$$
\end{description}
\end{thm}

\begin{prf} (i)  For $j,k \in J$, the operators 
$E_{jk} := \la \cdot, e_k\ra e_j$ are eigenvectors for the 
adjoint action of the diagonal subgroup 
$T \subeq \U_\infty(\cH)$ introduced above. The corresponding 
characters $t \mapsto t_j t_k^{-1}$ are pairwise different, so that the 
set of $T$-finite elements of $\gl(\cH)$ is the Lie 
subalgebra 
\[ \gl(J,\C) := \Spann \{ E_{jk} \: j,k \in J\}.\]  
Let $\leq$ denote a linear order on $J$ for which 
$\lambda$ is non-increasing and write 
$\Delta^+ := \{ \eps_j - \eps_k\: j < k\}$ for the 
corresponding positive system. 
Then $E_{jk}$, $j < k$, are positive root vectors  
and $v_\lambda \in \cH_\lambda$ is a highest weight 
vector for $\gl(J,\C)$ with respect to $\Delta^+$ 
(cf.~\cite[Prop.~II.1]{Ne98}). Hence 
$\dd\tilde\pi_\lambda^\C(E_{jk})v_\lambda = 0$ 
for $j < k$ and 
$\dd\tilde\pi_\lambda^\C(E_{jj})v_\lambda = \lambda_j v_\lambda$. 

The operator $E_{jk}$ is contained in $\fp_\cF^*$ 
if and only if $\lambda_j \geq \lambda_k$, which implies $j \leq k$. 
Hence the preceding observations prove (i) for $X \in \gl(J,\C) \cap 
\fp_\cF^*$. The general case now follows from the fact that both sides of 
\eqref{eq:eigenvect} are continuous with 
respect to the strong operator topology on $\gl(\cH)$ 
with respect to which $\gl(J,\C)$ is dense in 
$\gl(\cH)$. Actually $X = \sum_{j,k \in J} \la Xe_k, e_j\ra E_{jk}$ 
converges strongly. 

For $X \in \cN_\lambda$ we finally  obtain 
$\la \dd\tilde\pi_\lambda^\C(X)v_\lambda, v_\lambda \ra 
= \la v_\lambda, \dd\tilde\pi_\lambda^\C(X^*)v_\lambda \ra =0$. 

(ii) Since the subgroup $P_\lambda \subeq \GL(\cH)$ is connected, 
(i) implies that $v_\lambda$ is an eigenvector for the group  
\[ P_\lambda^* := \{ g^* \: g \in P_\lambda\} \] 
with Lie algebra $\fp_\cF = \cN_\lambda \oplus \{ D_\lambda\}'$.
Therefore (ii) follows from \cite[Thm.~5.11]{Ne10}. 

(iii) For every $X\in\fu(\cH)$ we have 
$ \Phi_{\widetilde{\pi}_\lambda}([v_\lambda])(X) 
=-i\la \dd\widetilde{\pi}_\lambda(X)v_\lambda,v_\lambda\ra $.  
The equality $\Phi_{\widetilde{\pi}_\lambda}([v_\lambda])=\psi_\lambda\vert_{\fu(\cH)}$ 
follows by (i) and 
$\cN_\lambda\oplus\{D_\lambda\}'\oplus\cN_\lambda^*=B(\cH)$. 

In view of (ii), the equality 
$I_{\widetilde{\pi}_\lambda}=\oline{\conv}^{w^*} (\cO_{\psi_\lambda})$ 
now follows from \cite[Thm.~5.11(c)]{Ne10}. 

As the momentum map $\Phi_{\pi_\lambda}$ is $\U(\cH)$-equivariant 
and the stabilizer of $\psi_\lambda$ is $\U(\cH)_\lambda$, 
the stabilizer of $[v_\lambda]$, resp., 
$[\alpha_{\lambda}]$ in $\U(\cH)$ coincides with $\U(\cH)_\lambda$, 
so that $\U(\cH)[v_\lambda] \cong \U(\cH)/\U(\cH)_\lambda 
\cong \cO_{\psi_\lambda}$. 

It remains to show that the complex structure on 
$\cO_{\psi_\lambda}$, together with the canonical symplectic form, 
defines a strong K\"ahler structure, i.e., 
that the corresponding quadratic form on $T_{\psi_\lambda}(\cO_{\psi_\lambda})$ 
is positive definite and defines a complete metric. 
For the complex structure $I$ on $T_{\psi_\lambda}(\cO_{\psi_\lambda}) 
\cong \cN_\lambda^*$ we obtain for 
$X = Z - Z^*$, 
$Z = \sum_{\lambda_j > \lambda_k} z_{jk} E_{jk} \in \cN_\lambda^*$: 
\begin{align*}
\psi_\lambda([X,IX])
&= \psi_\lambda([Z - Z^*, i(Z + Z^*)] 
= 2i \psi_\lambda([Z,Z^*]) = 2 \Tr(D_\lambda[Z,Z^*]) \\
&= 2 \sum_{\lambda_j > \lambda_k} |z_{jk}|^2 \Tr(D_\lambda [E_{jk}, E_{kj}]) 
= 2 \sum_{\lambda_j > \lambda_k} |z_{jk}|^2 (\lambda_j - \lambda_k) \\ 
&\geq 2\sum_{\lambda_j > \lambda_k} |z_{jk}|^2  \geq 0.   
\end{align*}
This shows that $T_{\psi_\lambda}(\cO_{\psi_\lambda})$ 
is complete with respect to the Hilbert norm \break $\Psi_\lambda([X,IX])^{1/2}$, 
and this means that the K\"ahler structure on 
$\cO_{\psi_\lambda}$ is strong. 
\end{prf}

In connection with Theorem~\ref{flags}, we note that a different description of the geometric realization can be obtained by the approach of \cite{BG08} 
that uses reproducing kernels on line bundles. 

\begin{rem} \mlabel{rem:2.7} 
Let $I_\lambda^2 \subeq \fu_2(\cH)' \cong \fu_2(\cH)$ denote the 
closure of $I_\lambda^\bn$ with respect to the Hilbert--Schmidt norm. 
Then $I_\lambda^2$ is a closed convex $\U(\cH)$-invariant subset 
of the Hilbert space $\fu_2(\cH)'$ containing 
$I_\lambda^\bn$. Since $\U(\cH)$ acts by isometries of the 
Hilbert--Schmidt norm and closed balls in a Hilbert space are strictly 
convex, the orbit $\cO_{\psi_\lambda}$ is contained 
in the set of extreme points of $I_\lambda^2$, and this leads immediately to 
\[ \cO_{\psi_\lambda} \subeq \Ext(I_\lambda^\bn).\] 
The same argument even implies that 
$\cO_{\psi_\lambda}$ consists of exposed points of $I_\lambda^\bn$ 
because each element of the 
orbit defines a norm continuous linear functional on $\fu_\infty(\cH)'$. 

On any bounded subset of $\fu_2(\cH)'$, the weak-$*$-topology 
coincides with the weak-$*$-topology defined by the space 
$F(\cH) \cap \fu(\cH)$ of skew-hermitian finite rank operators. 
Hence the weak-$*$-closure $I_\lambda$ of $I_\lambda^\bn$ in 
$\fu_2(\cH)'$ is also contained in $I_\lambda^2$, and we obtain 
with the same argument as above the stronger assertion  
\[ \cO_{\psi_\lambda} \subeq \Ext(I_\lambda).\] 

If $\mu \in I_\lambda\subeq I_\lambda^2$ 
satisfies $p_\ft(\mu) = \lambda$, 
then the fact that the projection onto diagonal operators 
is orthogonal in $\fu_2(\cH)'$ implies that 
$\|\mu\|_2 > \|\lambda\|_2$ whenever $\mu \not=\lambda$. 
On the other hand $\|\mu\|_2 \leq \|\lambda\|_2$ for every 
$\mu \in I_\lambda^2$, so that 
\[ p_\ft^{-1}(\psi_\lambda) \cap I_\lambda = \{\psi_\lambda\}.\] 
\end{rem}

\begin{problem}
%
(a) Does the group $G = \U_\infty(\cH)$ act transitively on 
the subset $\Phi_{\pi_\lambda}^{-1}(\Ext(I_\lambda))$ of 
$\bP(\cH_\lambda)$? 

(b) Does $\bP(\cH_\lambda)$ contain a unique complex $G$-orbit? 
This property is satisfied for irreducible representations 
of compact groups (cf.\ \cite[Ch.~XV]{Ne00}). 
Here Corollary~\ref{use-cor1} may be helpful.  
A natural first 
step may be to reduce the problem to $T$-eigenvectors 
by observing that  every $\U(\cH)$-orbit in $\bP(\cH_\lambda)$  
contains an element which is mapped into $\ft'$ by the momentum map. 
\end{problem}

%
%
%

\begin{rem}\label{proj}
 (a) For $X \in \Herm_1(\cH)$ we write 
$\Lambda^k(X)$ for the corresponding operator on 
$\Lambda^k(\cH)$, considered as a representation of the Lie 
algebra $\fS_1(\cH)$ (cf.\ \cite{Ne98}). If 
$X = \sum_j x_j E_{jj}$ is diagonal, then 
\[ s_k(X) := \supp\Spec(\Lambda^k(X)) = L_k((x_j)),\] 
so that the weak-$*$ lower semicontinuous convex function 
$s_k$ on $\Herm_1(\cH)$ is the unique $\U(\cH)$-invariant 
functional whose restriction to the diagonal is given by~$L_k$. 

Since the momentum set $I_\lambda$ 
is invariant under conjugation with 
$\U(\cH)$, it is determined by the diagonal operators it contains. 
We know already that 
\[ p_\ft(I_\lambda) = \co(\lambda) = I_{\pi_\lambda\res_T},\]  
the fact that $\lambda \in \ft'$ entails that 
$I_\lambda \cap \ft \supeq \co(\lambda).$  
We therefore obtain with Lemma~\ref{lem:convhul}(i)
\[ i I_\lambda 
= \{ X \in \Herm_1(\cH) \: 
(\forall k)\, 
s_k(X) \leq s_k(D_\lambda), 
s_k(-X) \leq s_k(-D_\lambda)\}.\] 

(b) For the norm-closed momentum set, we obtain the 
additional necessary 
condition $\tr(X) = \tr(D_\lambda) = \sum_j \lambda_j$. 
Since $p_\ft$ is a contraction with respect to the trace norm 
(it is the fixed point projection for the action of the compact 
group $\T^J$ of diagonal operators by conjugation), we find that 
\[ p_\ft(I_\lambda^\bn) \subeq I_{\pi_\lambda\res_T}^\bn.\] 
If a unit vector $v \in \cH_\lambda$ is written as a sum 
$v = \sum_\alpha v_\alpha$ of $T$-eigenvectors, we obtain 
\[ \Phi_{\pi_\lambda\res_T}([v]) 
= \sum_\alpha \|v_\alpha\|^2 \Phi_{\pi_\lambda\res_T}([v_\alpha]) 
= -i \sum_\alpha \|v_\alpha\|^2 \alpha.\] 
Since the weight set $\cP_\lambda$ is norm bounded, this series 
converges in norm, which leads to 
\[ I_{\pi_\lambda\res_T}^\bn = -i \co^\bn(\lambda).\] 
Taking into account that the right hand side is contained in $I_\lambda^\bn$, 
we obtain the equalities 
\[ p_\ft(I_\lambda^\bn) = I_{\pi_\lambda\res_T}^\bn = -i\co^\bn(\lambda) 
= I_\lambda^\bn \cap \ft. \] 
With Lemma~\ref{lem:convhul}(ii) we now see that 
\[ i I_\lambda^\bn  
= \{ X \in \Herm_1(\cH) \: 
(\forall k)\, 
s_k(\pm X) \leq s_k(\pm D_\lambda), \tr(X) = \tr(D_\lambda)\}\] 
because both sets have the same intersection with $\ft' \cong\ell^1(J,\R)$. 
\end{rem}

\begin{remark} 
In the same spirit, one finds an explicit description of the 
corresponding support functional 
\[ s_\lambda(X) := \sup(-i\dd\pi_\lambda(X)) 
= \sup\la I_\lambda, X \ra. \] 
These are invariant continuous, positively homogeneous convex 
functions, hence in particular determined by their 
values on diagonal operators, where one should try to find a more 
explicit formula. For the fundamental representation  
on $\Lambda^k(\cH)$ we have $s_\lambda = s_k$. 

Let $\leq$ be a linear order on $J$ for which $\lambda$ is non-increasing. 
If $X = \sum_j x_j E_{jj}$ is diagonal of finite rank, then $X$ 
is $\cW$-conjugate to an element representing a non-increasing function 
$J \to \R,j \mapsto x_j$. Then the relations 
\[ \cW\lambda \subeq \lambda - C_\lambda \quad \mbox{ and } \quad 
\la C_\lambda, X \ra \geq 0 \] 
imply that 
\[ \lambda(X) = \sup \la \cP_\lambda, X \ra = s_\lambda(X).\] 
Using the weak continuity of $s_\lambda$, this formula provides a constructive 
way to calculate $s_\lambda$ on every $X \in \fu(\cH)$ via approximation 
by finite rank operators. 
\end{remark}

\section{From $\U_\infty(\cH)$ to $C^*$-algebras}\label{Sect3}

In this section we describe the classes of unitary irreducible representations for the unitary groups of any unital $C^*$-algebra which occur in the Schur--Weyl decompositions 
of the tensor product representations 
(see Definition~\ref{the_repres}) and investigate their momentum sets. 
The information we obtain in Theorem~\ref{flags2} on momentum sets is sufficiently rich to allow us to conclude that these irreducible representations can be distinguished from each other by using their norm-closed momentum sets. 
Another important feature is that the extreme points of any of these momentum sets is a coadjoint orbit corresponding to the highest weight, just as in the case of 
finite dimensional Lie groups 
(see \cite{Ne00} and \cite{Wi92}). 

Throughout this section, $\cA$ denotes a unital $C^*$-algebra 
and $(\pi, \cH)$ is an irreducible $*$-representation of $\cA$. 

\begin{defn}\label{the_repres} 
For $\lambda \in \cP$, 
we write 
$(\tilde\pi_\lambda, \cH_\lambda)$ 
for  the unique unitary representation of $\U(\cH)$ 
extending the representation $\pi_\lambda$ of $\U_\infty(\cH)$ 
(\cite[Cor.~III.11 and Thm.~III.15]{Ne98}). 
This leads to a unitary representation 
\[ \pi_\lambda^{\cA}:= \tilde\pi_\lambda \circ \pi\vert_{\U(\cA)}
\colon \U(\cA)\to \U(\cH_\lambda).\] 
\end{defn}

From \cite[Thm.~2.8.3(iii)]{Dix64} we obtain: 
\begin{prop} \mlabel{prop:2.1} 
For each finite dimensional subspace 
$F \subeq \cH$ and each unitary operator 
$u \in \U(F)$, there exists a $g \in \U(\cA)$ with 
$\pi(g)\res_F  = u.$ In particular, $\pi(g)(F) \subeq F$. 
\end{prop}

\begin{cor}\label{cor_dense}
The group $\pi(\U(\cA))$ is strongly dense in $\U(\cH)$. 
\end{cor}

\begin{thm}\mlabel{thm:1.3} 
For $\lambda \in \cP$, the representation $\pi_{\lambda}^\cA$ 
of $\U(\cA)$ is irreducible and two such representations 
$\pi_\lambda^\cA$ and $\pi_\mu^\cA$ are equivalent if and only if 
$\mu \in \cW\lambda$.
\end{thm} 

\begin{prf} Recall $n_\lambda = \sum_{j \in J}\lambda_j$. 
Since the representation of 
$\U(\cH)$ on 
\[ \cH^{\otimes n_{\lambda_+}} \otimes (\cH^*)^{\otimes n_{\lambda_-}} 
\supeq \cH_\lambda \] 
(Theorem~\ref{thm:parti}) 
is continuous with respect to the strong operator topology 
on $\U(\cH)$, the subgroup $\pi_\lambda^\cA(\U(\cA))$ 
is strongly dense in $\pi_\lambda(\U(\cH))$ (Corollary~\ref{cor_dense}), 
hence has the same commutant. 
Therefore $(\pi_\lambda^\cA, \cH_\lambda)$ is irreducible. 

The same argument implies that the representations 
$\pi_\lambda$ and $\pi_\mu$ of $\U(\cH)$ and the 
corresponding representations $\pi_\lambda^\cA$ and $\pi_\mu^\cA$ 
of $\U(\cA)$ define the same set 
$B_{\U(\cA)}(\cH_\lambda, \cH_\mu) = B_{\U(\cH)}(\cH_\lambda, \cH_\mu)$ 
of intertwining operators. We conclude that 
$\pi_\lambda^\cA \cong \pi_\mu^\cA$ is equivalent to 
$\pi_\lambda \cong \pi_\mu$, which corresponds to 
$\mu \in \cW\lambda$ (Theorem~\ref{thm:parti}). 
\end{prf}


The main point of the following theorem is that 
the representations $\pi^\cA_\lambda$ 
can be distinguished from each other by their norm-closed momentum sets. 

\begin{thm}\label{flags2} 
Let $(\pi, \cH)$ be an irreducible representation of the 
unital $C^*$-algebra $\cA$ on $\cH \cong \ell^2(J,\C)$ 
and $\lambda \in \cP$. 
Then the following assertions hold for the 
unitary representation $\pi_\lambda^\cA$ of $\U(\cA)$: 
\begin{description}
\item[\rm(i)] The unitary subgroup $\pi(\U(\cA))$ of $\U(\cH)$ 
acts transitively on the coadjoint orbit $\cO_{\psi_\lambda}$. 
Let $\psi_\lambda^{\cA}:=\psi_\lambda\circ\pi$ 
and write $\cO_{\psi_\lambda^{\cA}}$ for the coadjoint $\U(\cA)$-orbit 
of $\psi_\lambda^{\cA}\vert_{\fu(\cA)}$. 
Then the momentum set of $\pi_\lambda^\cA$ is given by 
\[ \Phi_{\pi_\lambda^{\cA}}([v_\lambda])=\psi_\lambda^{\cA}\vert_{\fu(\cA)} 
\quad\text{ and }\quad
I_{\pi_\lambda^{\cA}}=\oline{\conv}^{w^*} \big(\cO_{\psi_\lambda^{\cA}}\big).\] 
\item[\rm(ii)] The coadjoint orbit $\cO_{\psi_\lambda^{\cA}}$ is norm closed in $\fu(\cA)'$ and is contained in the sphere centered at $0$ with radius $\Vert\lambda\Vert_1$. 
Moreover we have 
\begin{equation}\label{extreme}
\cO_{\psi_\lambda^{\cA}}=\Inm_{\pi_\lambda^{\cA}}\cap\Ext(I_{\pi_\lambda^{\cA}})=\Ext(\Inm_{\pi_\lambda^{\cA}}).
\end{equation} 
\item[\rm(iii)] For two such representations $\pi_\lambda^\cA$ 
and $\pi_\mu^\cA$, we have 
$$\Inm_{\pi_\lambda}=\Inm_{\pi_\mu}
\iff
\cO_{\psi_\lambda^{\cA}}=\cO_{\psi_\mu^{\cA}}
\iff 
\cW \lambda= \cW \mu
\iff \pi_\lambda^{\cA}\simeq \pi_\mu^{\cA}.$$
\end{description}
\end{thm}

\begin{prf} (i) In the notation of the proof of Theorem~\ref{flags}, 
the coadjoint orbit $\cO_{\psi_\lambda} \cong \U(\cH)/\U(\cH)_\lambda$ 
can be identified with the flag manifold $\Gr_\cF$. 
It therefore suffices to show that $\pi(\U(\cA))$ 
acts transitively on $\Gr_\cF$. 
An application of Proposition~\ref{prop:2.1} (see also \cite[Thm.~1]{GK60}) 
shows that for $u\in \U(\cH)$ 
there exists $g\in\pi(\U(\cA))$ such that $g=u$ on the finite 
dimensional subspaces 
\[ F_{-N} \subeq \cdots \subeq F_{-1}, \quad 
F_0^\bot \supeq F_1^\bot \supeq \ldots \supeq F_N^\bot. \] 
Since both $g$ and $u$ are unitary operators, it then follows that 
$g\cF = u\cF$. 

To prove the second part of the assertion, we first note that 
the adjoint map 
$(\pi\vert_{\fu(\cA)})'\colon\fu(\cH)'\to\fu(\cA)'$ 
satisfies 
$(\pi\vert_{\fu(\cA)})'I_{\widetilde{\pi}_\lambda}=I_{\pi_\lambda^{\cA}}$ 
(Proposition~\ref{pullback}). 
It then follows by Theorem~\ref{flags}(ii) that 
\begin{equation}\label{eq:convhull}
I_{\pi_\lambda^{\cA}}=(\pi\vert_{\fu(\cA)})'\oline{\conv}^{w^*}
\big(\cO_{\psi_\lambda}\big) 
\subseteq \oline{\conv}^{w^*}
\big((\pi\vert_{\fu(\cA)})' (\cO_{\psi_\lambda})\big).
\end{equation}
On the other hand, we have seen above that 
$\pi(\U(\cA))$ acts transitively on $\cO_{\psi_\lambda}$, 
and then it is straightforward to check that 
$(\pi\vert_{\fu(\cA)})' \cO_{\psi_\lambda}=\cO_{\psi_\lambda^{\cA}}$. 
Therefore 
$I_{\pi_\lambda^{\cA}}\subseteq
\oline{\conv}^{w^*}\big(\cO_{\psi_\lambda^{\cA}}\big)$. 

For the converse inclusion note that the equality 
$\pi_\lambda^{\cA}=\widetilde{\pi}_\lambda\circ\pi\vert_{\U(\cA)}$ implies 
$\Phi_{\pi_\lambda^{\cA}}=(\pi\vert_{\fu(\cA)})'\circ \Phi_{\widetilde{\pi}_\lambda}$. 
Since (i) ensures that $\Phi_{\widetilde{\pi}_\lambda}([v_\lambda])=\psi_\lambda\vert_{\fu(\cH)}$, 
it then follows that $\Phi_{\pi_\lambda^{\cA}}([v_\lambda])=\psi_\lambda^{\cA}\vert_{\fu(\cA)}$. 
Now, by using the $\U(\cA)$-equivariance property of the momentum map 
$\Phi_{\pi_\lambda^{\cA}}$, we get the converse inequality 
$I_{\pi_\lambda^{\cA}}\supeq\oline{\conv}^{w^*}(\cO_{\psi_\lambda^{\cA}})$, 
and this proves (i). 

(ii) The natural isometric isomorphism 
$\fu(\cH)_*\simeq\Herm_1(\cH)$ takes $\cO_{\psi_\lambda}$ 
onto the unitary equivalence $\U(\cH)$-orbit $\cO_{D_\lambda}$ of the finite-rank operator $D_\lambda$. 
Since $D_\lambda$ is a self-adjoint finite-rank operator, 
it generates a finite dimensional $C^*$-algebra, and then it is well known that $\cO_{D_\lambda}$ is norm-closed in $B(\cH)$; 
see for instance \cite[Prop.~2.4]{Vo76}. 
Then $\cO_{D_\lambda}$ is in particular a closed subset of the Banach space 
$\Herm_1(\cH)$. 
(In the case $\lambda\ge0$, this was also noted in \cite[Prop.~3.1(iv)]{AK06}.)
Also note that $\cO_{D_\lambda}$ coincides with the unitary equivalence $\U(\cA)$-orbit of the finite-rank operator $D_\lambda$ by the transitivity theorem \cite[Thm.~1]{GK60}. 

Now apply Proposition~\ref{dual_appl} for $\cM=B(\cH)$ and the normal representation $\widetilde{\pi}_\lambda\colon\U(\cH)\to\U(\cH_\lambda)$. 
It then follows that the $\U(\cA)$-orbit $\cO_{\psi_\lambda^{\cA}}$ 
is the image of the norm-closed set $\cO_{D_\lambda}$ 
by an isometric isomorphism $R_{\fu(\cA)}\colon\fu(\cM)_*\to\fu(\cA)^*$, 
and $\fu(\cA)^*$ is a closed linear subspace of the topological dual $\fu(\cA)'$. 
Therefore $\cO_{\psi_\lambda^{\cA}}$ is norm closed in $\fu(\cA)'$. 
The assertion on the sphere containing $\cO_{\psi_\lambda^{\cA}}$ follows since the aforementioned isometric isomorphism takes $\psi_\lambda^{\cA}$ 
to $D_\lambda$ and $\Vert D_\lambda\Vert_1=\Vert\lambda\Vert_1$. 

In order to prove \eqref{extreme}, note that, due to Proposition~\ref{dual_appl} and to the isometric isomorphism used above, it actually suffices to show that 
\begin{equation}\label{extreme1}
\cO_{D_\lambda}=I_\lambda^\bn\cap\Ext(I_\lambda)=\Ext(I_\lambda^\bn).
\end{equation}
where we use the notation $I_\lambda=I_{\pi_\lambda}$ 
and $I_\lambda^\bn=\Inm_{\pi_\lambda}=\Inm_{\widetilde{\pi}_\lambda}
\subset-i\Herm_1(\cH)$ as in Remark~\ref{proj}. 
The inclusion $\cO_{D_\lambda}\subeq I_\lambda^\bn\cap\Ext(I_\lambda)$ 
follows by Remark~\ref{rem:2.7} and it is clear that 
$I_\lambda^\bn\cap\Ext(I_\lambda)\subeq\Ext(I_\lambda^\bn)$. 
To prove that $\Ext(I_\lambda^\bn)\subeq\cO_{D_\lambda}$, let $X\in \Ext(I_\lambda^\bn)$ arbitrary. 
Since the set $\Ext(I_\lambda^\bn)$ is naturally acted on by the full unitary group $\U(\cH)$ and $X\in\Herm_1(\cH)$, we may assume that 
$X\in I_\lambda^\bn\cap\ft$, hence 
$X\in\Ext(I_\lambda^\bn\cap\ft)$. 
Then Remark~\ref{proj} shows that $iX\in\Ext(\co^\bn(\lambda))$, 
and now $X\in\cO_{D_\lambda}$ as a consequence of Proposition~\ref{obw}.

(iii) We know from Theorem~\ref{thm:1.3} that 
the representations $\pi_\lambda^{\cA}$ and 
$\pi_\mu^{\cA}$ are equivalent if and only if $\cW\lambda=\cW\mu$, 
and if this is the case, then $\cO_{\psi_\lambda^{\cA}}=\cO_{\psi_\mu^{\cA}}$. 
Now let us assume that the latter equality of coadjoint $\U(\cA)$-orbits holds. 
Then there exists $u\in \U(\cA)$ such that 
$\psi_\lambda=\psi_\mu\circ\Ad_{\fu(\cH)}(\pi(u))$ on $\pi(\fu(\cA))$. 
Since $\pi(\cA)$ is dense in $B(\cH)$ in the weak operator topology 
and both sides of the above equality are continuous with respect to 
this topology, it follows at once that $\psi_\lambda=\psi_\mu\circ\Ad_{\fu(\cH)}(\pi(u))$ on $B(\cH)$, 
which leads to $D_\lambda=\pi(u)^{-1}D_\mu\pi(u)$. 
Both $D_\lambda$ and $D_\mu$ are self-adjoint diagonal operators, 
with the spectra 
(including the spectral multiplicities) described by the functions 
$\lambda, \mu \colon J\to\Z$. That $D_\lambda$ and $D_\mu$ are conjugate
implies that, for every $n \in \Z$, 
we have $|\lambda^{-1}(n)| = |\mu^{-1}(n)|$, 
and since these numbers are finite for $n \not=0$, 
we obtain $\mu \in \cW\lambda$ 
(cf.\ Theorem~\ref{thm:parti}). 

Finally, if we have  
$\Inm_{\pi_\lambda^{\cA}}=\Inm_{\pi_\mu^{\cA}}$, then (ii) above 
shows that $\cO_{\psi_\lambda^{\cA}}=\cO_{\psi_\mu^{\cA}}$, and this completes the proof.
\end{prf}

\begin{rem} If $\dim \cH = n < \infty$ and $(\pi, \cH)$ is a finite 
dimensional representation of $\cA$, then 
$\pi(\cA) = B(\cH) = K(\cH)$, so that the results of this 
section are trivial consequences of the corresponding ones 
for irreducible unitary representations of $\U_n(\C)$. 
\end{rem}

\begin{rem} An irreducible representation 
$(\pi, \cH)$ of the $C^*$-algebra can be obtained with the 
GNS construction from any pure state $\phi$ of the form 
\[ \phi(A)  = \la \pi(A)v,v\ra, \quad \|v\| = 1, v \in \cH\] 
and since $\pi(\U(\cA))$ acts transitively on $\bP(\cH)$, 
the pure states defining equivalent representations form 
a $\U(\cA)$-orbit in the state space $S(\cA)$. The 
map 
\[ \bP(\cH) \to S(\cA),\quad 
[v] \mapsto \phi_v, \quad \phi_v(A) 
:= \frac{\la \pi(A)v,v\ra}{\la v,v\ra} \] 
is injective because $\pi(\cA) \subeq B(\cH)$ is dense in the 
weak operator topology. 

It is instructive to describe the functional 
$i\psi_\lambda^\cA$ in $C^*$-algebraic terms.  
To obtain such a description, we call an $n$-tupel 
$(\phi_1,\ldots, \phi_n) \in \U(\cA)\phi \subeq S(\cA)$ 
{\it orthogonal} if 
$\phi_j = \phi_{v_j}$ for pairwise orthogonal elements 
$v_1,\ldots, v_n \in \cH$. 
Note that Proposition~\ref{prop:2.1} implies that 
$\U(\cA)$ acts transitively on the orthogonal 
$n$-tupels in the orbit $\U(\cA)
\phi$.
\begin{footnote}
  {For any orthogonal $n$-tuple of states, the functional 
$\phi := \frac{1}{n}(\phi_1 + \cdots + \phi_n)$ also is a state. 
It corresponds to the element $(v_1,\ldots, v_n) \in \cH^n$, which is a 
cyclic vector for the natural representation of $\cA$ on 
$\cH^n$ (\cite[Thm.~2.8.3(iii)]{Dix64}; Proposition~\ref{prop:2.1}).}
\end{footnote}

Writing 
$\lambda = \sum_{j \in J} \lambda_j \eps_j$ 
as a finite sum, where $\eps_j(k) = \delta_{jk}$,  
the vectors $\{ e_j \: j \in \supp(\lambda)\}$ 
form an orthonormal set corresponding to the 
orthogonal states $\phi_j$, $j \in\supp(\lambda)$.  We 
now have 
\[ i\psi_\lambda^\cA(A) 
= \trace(D_\lambda \pi(A)) 
= \sum_{j \in J} \lambda_j \la \pi(A)e_j, e_j\ra 
= \sum_{j \in J} \lambda_j \phi_j(A),\] 
so that 
\[ \psi_\lambda^\cA = -i \sum_{j \in \supp(\lambda)} \lambda_j \phi_j\]  
is an intrinsic description of $\psi_\lambda^\cA$ in terms 
of orthogonal states in $\U(\cA)\phi$. 
Fixing a bijection $\gamma \: \{ 1,\ldots, N\} \to \supp(\lambda)$, 
it now follows that the coadjoint 
orbit $\cO_{\psi_\lambda^\cA}$ consists of the restrictions of 
all functional of the form 
\[ -i \sum_{j = 1}^N \lambda_{\gamma(j)} \phi_j, \] 
where $(\phi_1, \ldots, \phi_N)$ is an orthogonal 
$N$-tupel in $\U(\cA)\phi$. 

For the particular case where 
$\lambda_j \in \{0,1\}$, we have $\cH_\lambda \cong \Lambda^N(\cH)$ 
and the elements of $\cO_{\psi_\lambda^\cA}$ correspond to sums 
of orthogonal $N$-tuples in $\U(\cA)\phi$. 

For the case $N = 1$ and $\lambda = m \delta_{j_0}$ we 
simply obtain $\cO_{\psi_\lambda^\cA} = -i m\U(\cA)\phi$. 
\end{rem} 

\begin{rem}
In the setting of Theorem~\ref{flags2}(i), 
it is not clear whether the coadjoint orbit 
$\cO_{\psi_\lambda^{\cA}}\subseteq\fu(\cA)'$ 
is ``smooth'' in the sense that the isotropy Lie algebra at any 
of its points is a complemented subspace of $\fu(\cA)$. 
Nevertheless, if $K(\cH)\subseteq\pi(\cA)$, then the fact that the 
natural complement of $\fu(\cH)_\lambda$ can be chosen in 
$\fu_\infty(\cH)$ implies that it also is a complement 
of the stabilizer algebra in $\pi(\fu(\cA))$. 
\end{rem}

We can now describe which ones of the representations $\pi_\lambda^{\cA}$ are contractive in the sense of \cite{Pa83} and \cite{Pa87}. 

\begin{prop}\label{contr_SW}
If $\lambda\in\cP$, then we have 
\begin{equation}
  \label{eq:esti}
(\forall g_1,g_2\in\U(\cA))\quad 
\Vert\pi_\lambda^{\cA}(g_1)-\pi_\lambda^{\cA}(g_2)\Vert\le\Vert g_1-g_2\Vert 
\end{equation}
if and only if $\lambda = \pm \eps_j$ for some $j \in J$, i.e., 
$\pi_\lambda^{\cA}$ is either the representation $\pi\vert_{\U(\cA)}$ or its dual.
\end{prop}

\begin{prf}
If \eqref{eq:esti} 
is satisfied, then it follows by Lemma~\ref{contr} that for every $\mu\in I_{\pi_\lambda^{\cA}}$ 
we have $\Vert\mu\Vert\le 1$. 
On the other hand $\psi_\lambda^{\cA}\in I_{\pi_\lambda^{\cA}}$ (see Theorem~\ref{flags2}), 
hence we get by Theorem~\ref{flags2}(ii) that $\Vert\lambda\Vert_1\le1$. 
Since $\lambda$ is an integer valued function on $J$, the assertion follows. 
\end{prf}

\subsection*
{Using the momentum sets in classification problems}


We have seen in Theorem~\ref{flags2}(iii) that, fixing the 
algebra representation $\pi$, the 
unitary representations of the form $\pi_\lambda^{\cA}$ 
can be distinguished by the corresponding 
coadjoint orbits $\cO_{\psi_\lambda^{\cA}}$ and the norm closed 
momentum sets $I_{\pi_\lambda}^{\bf n}$. 
A priori, the coadjoint orbit is not defined intrinsically 
in terms of the representation $\pi_\lambda^\cA$ of $\U(\cA)$, 
but $I_{\pi_\lambda}^{\bf n}$ is. 
Nevertheless, just as in the representation 
theory of compact Lie groups (see \cite{Wi92} and \cite{Ne00}), 
we found that 
$\cO_{\psi_\lambda^\cA}$ can be specified intrinsically 
as the set of extreme points of $I_{\pi_\lambda}^{\bf n}$. 
%

Another intrinsically defined object is the full momentum set 
$I_{\pi_\lambda}^\cA$. As we have seen in \cite[Thm.~7.1]{Ne10}, 
it does not separate the unitary representations of $\U(\cA)$ 
obtained by restricting inequivalent algebra representations 
with the same kernel, and such representations exist 
for separable $C^*$-algebras not of type I 
(\cite[Thm.~9.1]{Dix64}, \cite{Sa67}). 

\begin{rem} \label{rem:3.9} 
(a) However, \cite[Thm.~1.1]{KOS03} 
asserts that the normal subgroup of asymptotically inner automorphisms of 
a separable $C^*$-algebra 
$\cA$ acts transitively on $\Ext(I_\pi)$ for any irreducible 
representation $\pi$ of $\cA$. Indeed, $\Ext(I_\pi)$ is the set of 
all pure states $\phi$ of $\cA$ for which the corresponding representation 
$\pi_\phi$ has the same  kernel as $\pi$, hence the same momentum set 
(cf.\ \cite[Thm.~X.5.12]{Ne00}). 

(b) According to \cite[Prop.~5.1.3]{Dix64}, two representations 
$\pi_1$ and $\pi_2$ of a $C^*$-algebra $\cA$ are quasi-equivalent 
if and only if they have equivalent  multiples. For the corresponding 
unitary representations of $\U(\cA)$, this means that 
$I_{\pi_1}^\bn =I_{\pi_2}^\bn$. In fact, multiples of a given representations 
have the same norm-closed momentum set, so that 
quasi-equivalence of $\pi_1$ and $\pi_2$ implies equality of the norm-closed 
momentum sets. If, conversely, $I_{\pi_1}^\bn = I_{\pi_2}^\bn$, then the 
two representations $\pi_1$ and $\pi_2$ have the same set of 
normal states, so that their extensions 
$\pi_1^{**} \: \cA^{**} \to B(\cH_1)$ and 
$\pi_2^{**} \: \cA^{**} \to B(\cH_2)$ to the enveloping 
$W^*$-algebra $\cA^{**}$ have the same kernel. 
This in turn implies that 
\[ \pi_1(\cA)'' 
= \pi_1^{**}(\cA^{**})\cong 
\cA^{**}/\ker \pi_1^{**} \cong 
 \pi_2^{**}(\cA^{**})
= \pi_2(\cA)'',\] 
so that \cite[Prop.~5.1.3]{Dix64} implies that $\pi_1$ and $\pi_2$ 
are quasi-equivalent. 

(c) According to \cite[Prop.~5.3.3]{Dix64}, two irreducible 
representations of a $C^*$-algebra which are quasi-equivalent are 
equivalent. This means that in the class of those 
unitary representations of $\U(\cA)$ which are restrictions of algebra 
representations, the norm-closed momentum set $I_\pi^\bn$ determines 
the equivalence class of $\pi$. As we have already observed above, 
this is not true for the momentum set $I_\pi$ if $\cA$ is not of 
type $I$. 

Geometrically, the fact that $I_\pi^\bn$ determines the equivalence 
class of the unitary representation $(\pi,\cH)$ of $\U(\cA)$ coming 
from an algebra representation is that 
$I_\pi^\bn \cong S_*(B(\cH))$ and that the extreme points of this 
set are parametrized by the projective space $\bP(\cH)$ of 
one-dimensional subspaces of $\cH$. As $\U(\cA)$ acts transitively 
on this set, it acts transitively on the extreme points of 
$I_\pi^\bn$, so that $\Ext(I_\pi^\bn)$ is the $\U(\cA)$-orbit 
in $\Ext(S(\cA))$ consisting of all those pure states $\phi$ 
with $\pi_\phi \cong \pi$ as algebra representations. 
\end{rem}

In the present subsection we discuss some applications of the momentum sets to the problem of classifying representations of unitary groups 
for various classes of $C^*$-algebras. 
Before going any further in this direction, let us settle the case of 
$C^*$-algebras of type~I by a statement which extends Proposition~\ref{prop:mostbasic}. 

\begin{prop}\label{type_I}
Let $\pi\colon\cA\to B(\cH)$ be 
any $*$-representation of a unital $C^*$-algebra 
such that $K(\cH)\subseteq\pi(\cA)$. 
If $\lambda,\mu\in\cP$,  then  
$\pi_\lambda^{\cA}\cong \pi_\mu^{\cA}$ if and only if 
$I_{\pi_\lambda^{\cA}}=I_{\pi_\mu^{\cA}}$.
\end{prop}

\begin{proof} First we note that $K(\cH) \subeq \pi(\cA)$ implies that 
the representation $\pi$ is irreducible. 
Assume that $\pi_\lambda^{\cA}\not\cong \pi_\mu^{\cA}$. 
We have to show that $I_{\pi_\lambda^{\cA}}\not=I_{\pi_\mu^{\cA}}$.
In view of Theorem~\ref{thm:1.3}, $\mu\not\in \cW\lambda$, 
so that Proposition~\ref{prop:mostbasic}, combined with 
Theorem~\ref{thm:parti}, implies that the subsets $I_\lambda$ and $I_\mu$ 
of $\fu_\infty(\cH)'$ are different. 

Since $K(\cH)\subseteq\pi(\cA)$, 
Proposition~\ref{pullback} shows 
that $I_\lambda=I_{\pi_\lambda^{\cA}}|_{\fu_\infty(\cH)}$ 
and 
$I_\mu=I_{\pi_\mu^{\cA}}|_{\fu_\infty(\cH)}$. 
Hence $I_{\pi_\lambda^{\cA}}\ne I_{\pi_\mu^{\cA}}$. 
\end{proof}

\begin{cor}\label{type_I_cor}
Let $\pi\colon\cA\to B(\cH)$ be an irreducible $*$-representation of a unital $C^*$-algebra of type~I. 
If $\lambda,\mu\colon J\to\Z$ are non-decreasing, finitely supported functions,  then $\pi_\lambda^{\cA}\cong \pi_\mu^{\cA}$ if and only if 
$I_{\pi_\lambda^{\cA}}=I_{\pi_\mu^{\cA}}$.
\end{cor}

\begin{prf} To derive this from Proposition~\ref{type_I}, 
we recall that a $C^*$-algebra $\cA$ is of type $I$ if and only if 
for every irreducible representation 
$(\pi, \cH)$ we have $K(\cH) \supeq \pi(\cA)$ 
(cf.~\cite{Sa67}, where this property is called GCR). 
\end{prf}

\begin{prop} \mlabel{prop:idealdet} For $0\not=\lambda \in \cP$, 
the ideal $\ker \pi$ can be recovered from 
$I_{\pi_\lambda^\cA}$ as the unique largest ideal of 
$\cA$ contained in $I_{\pi_\lambda^\cA}^\bot$. 
\end{prop}

\begin{prf} We consider the subspace 
\[ \cB := \{ A \in \cA \: \la I_{\pi_\lambda^\cA}, A \ra = \{0\}\} 
= (\ker \dd\pi_\lambda^\cA)_\C. \] 
From 
\[ \ker \pi \cap \fu(\cA) \subeq 
\ker \dd\pi_\lambda^\cA = I_{\pi_\lambda^\cA}^\bot \] 
it immediately follows that $\cB \supeq \ker \pi$. 
Therefore it remains to show that $\pi(\cB)$ contains no non-zero 
ideal of $\pi(\cA)$. 

As $\fu_\infty(\cH)$ is a simple Banach--Lie algebra 
and $\pi_\lambda$ is non-trivial, 
$\ker \dd\pi_\lambda = \{0\}$. Therefore every 
element in $\ker \dd\tilde\pi_\lambda$ commutes with 
$\fu_\infty(\cH)$, hence is of the form $\zeta \1$ for 
some $\zeta\in \T$. For these elements we have 
\[ \tilde\pi_\lambda(\zeta \1) 
= \prod_{j \in J} \zeta^{\lambda_j}
= \zeta^{\sum_j \lambda_j}.\] 
Therefore $\ker\dd\tilde\pi_\lambda$ is non-zero 
if and only if $\sum_j \lambda_j = 0$, and in this case it 
coincides with the center $\T \1$. We conclude that 
$\ker \pi_\lambda^\cA \subeq \T \1$ and hence that 
$\pi(\cB) \subeq \C \1$. 

If $\pi(\cB) = \{0\}$, we have $\cB = \ker \pi$ and there is nothing to show. 
So we assume that $\pi(\cB) =\C \1$. 
Then $\pi(\cB)$ contains no non-zero ideal of $\pi(\cA)$, which 
in turn shows that every ideal of $\cA$ contained in $\cB$ is contained 
in $\ker \pi$. This means that $\ker \pi$ is the unique largest 
ideal of $\cA$ contained in $\cB$. 
\end{prf}

Since an irreducible representation of a separable type $I$ algebra 
is determined by its kernel (\cite[Thm.~9.1]{Dix64}), we immediately derive 
with Corollary~\ref{type_I_cor} and Proposition~\ref{prop:idealdet}: 

\begin{thm} \mlabel{thm:idealdet} For $\cA$ separable of type $I$, 
$0\not=\lambda, \mu \in \cP$ and two irreducible 
representations $(\pi, \cH)$ and $(\rho, \cK)$, we have 
\[ I_{\pi_\lambda^\cA} = I_{\rho_\mu^\cA} 
\quad \Rarrow \quad 
\pi \cong \rho \quad \mbox{ and } \quad 
\pi_\lambda^\cA \cong \rho_\mu^\cA.\] 
\end{thm}

\begin{ex} (a) The argument in Proposition~\ref{prop:idealdet} 
does not require that $\dim \cH = \infty$, it is only needed 
that $\pi(\cA) \not=\C \1$, i.e., that $\dim \cH > 1$. 

If $\dim \cH = 1$, then $\cW =\{\1\}$ and 
$\cP = \Z$, the character group of 
$\U_1(\C) \cong \T$. For $\pi(a) = \chi(a)\1$, we accordingly have 
$\pi_\lambda^\cA(a) = \chi(a)^\lambda\1$. 

(b) Suppose that $\cA$ is commutative, i.e., 
$\cA = C(X)$ for a compact space $X$. 
Then $\U(\cA) = C(X,\T)$, and by taking tensor products, 
the preceding construction leads to all characters of 
$C(X,\T)$ which are finite products of point evaluations. 
If $X$ is totally disconnected, then this exhausts the character 
group of $C(X,\T)$ (\cite{Au93}), but in general there are much more. 

In fact, for $X = [0,1]$, we have 
$C([0,1],\T) \cong C([0,1],\R)/\Z$, so that 
the character group can be identified with the set 
of real-valued Borel measures $\mu$ on $[0,1]$ 
with integral total mass. Clearly, all Dirac measures 
have this property, but Lebesgue measure also does.
\end{ex}

\begin{rem} \mlabel{rem:redux} Suppose that $\cA \subeq B(\cH)$ is a 
$C^*$-algebra containing $K(\cH)$. 

Recall that every irreducible unitary representation 
$(\pi_\lambda, \cH_\lambda)$ of $\U_\infty(\cA)$ 
extends to $\U(\cH) \supeq \U(\cA)$, so that 
$\U(\cA)$ acts trivially on the 
set of equivalence classes of irreducible unitary representations 
of the normal subgroup $\U_\infty(\cH)$. 

Let $(\pi, \cH)$ be an irreducible 
unitary representation of $\U(\cA)$. Then 
$\pi\res_{\U_\infty(\cH)}$ decomposes into irreducible 
representations, and the preceding remark implies that the 
isoptypic components are invariant under $\U(\cA)$. 
Therefore 
\[ \cH \cong \cM \otimes \cH_\lambda \quad \mbox{ and } \quad 
\pi = \gamma \otimes \pi_\lambda^\cA,\] 
 where 
$\gamma\: \U(\cA)/\U_\infty(\cH) \to \U(\cM)$ is irreducible. 

In this sense every irreducible representation of 
$\U(\cA)$ can be decomposed into one of the type $\pi_\lambda^\cA$ 
and a representation of the quotient group $\U(\cA)/\U_\infty(\cH)$, 
which is a subgroup of $\U(\cA/K(\cH))$, so that one may hope 
for an inductive description of irreducible representations 
if $\cA/K(\cH)$ again has a faithful irreducible representation 
whose image contains the compact operators. 
\end{rem}

\begin{ex} Let $\cA \subeq B(\ell^2)$ be the {\it Toeplitz algebra}, 
generated by the shift operator $S$ on $\ell^2 = \ell^2(\N,\C)$ 
and its adjoint. 
Then $\cA$ contains $K(\cH)$ and $\cA/K(\cH) \cong C(\bS^1)$ is 
commutative. In particular, we have a short exact sequence of 
Banach--Lie groups 
\[ \1 \to \U_\infty(\cH) \to \U(\cA) \to C(\bS^1,\T)_0 \to \1\] 
which implies that $\U(\cA)$ is connected. 

Since $\cA$ is separable and an extension of a commutative algebra 
by a type $I$ algebra, it is also of type $I$. Therefore 
each irreducible unitary representation is determined 
uniquely by its kernel. The simplicity of $K(\cH)$ 
implies that it is contained in every non-zero ideal 
$\cI$ of $\cA$, so that every irreducible representation 
with a non-trivial kernel factors through the commutative 
quotient $C(\bS^1)$. This implies that 
$\hat \cA = \{\id\} \dot\cup \bS^1$, i.e., there is only one 
infinite dimensional irreducible representation and all others 
are one-dimensional, parametrized by $\bS^1$. 

In this case we can also determine all irreducible 
unitary representations $(\pi, \cH)$ of $\U(\cA)$. 
If $\ker \pi$ contains $\U_\infty(\cH)$, the representation 
factors through the abelian quotient, hence is one-dimensional. 
These representations are para\-met\-rized by the characters of 
the Banach--Lie group $C(\bS^1,\T)_0 \cong \T \times C_*(\bS^1,\R)$, 
which is a product of $\T$ and the Banach space 
$C_*(\bS^1,\R)$. 

If $(\pi, \cH)$ is an irreducible continuous unitary 
representation of $\U(\cA)$, 
then Remark~\ref{rem:redux}  implies that 
$\cH \cong \cM \otimes \cH_\lambda$ and 
$\pi = \gamma \otimes \pi_\lambda^\cA$, where the representation
$\gamma\: \U(\cA)/\U_\infty(\cH) \to \U(\cM)$ is irreducible, 
hence one-dimensional. Therefore $\cH = \cH_\lambda$ and 
$\pi(g) = \chi(g) \pi_\lambda^\cA(g)$, where 
$\chi \: \U(\cA) \to \T$ is a character; actually a pull-back 
of a character of $C(\bS^1,\T)_0$. This provides a complete 
description of the continuous unitary representations 
of $\U(\cA)$. 
\end{ex}

\section{More on extreme points of momentum sets}\label{Sect4}

In this section we point out additional properties of the extreme points of momentum sets, by using some basic ideas of infinite dimensional convexity. 
We refer to \cite{FLP01} for a survey on convexity in Banach spaces. 
In particular, for arbitrary $\lambda\in\cP$, we are thus able to obtain 
in Corollary~\ref{answer} below an intrinsic description of the coadjoint orbit $\cO_{D_\lambda}$ as the set of weak-$*$-strongly exposed points of the momentum set $I_\lambda$. 
We recall that such a description involving  the norm-closed momentum set 
$I_\lambda^\bn$
was already obtained in Theorem~\ref{flags2} (see \eqref{extreme}).

\begin{defn}\label{ext_def}
If $\cX$ is a real Banach space with the topological dual $\cX'$, 
and $A\subeq\cX'$, then we define the following subsets of the set $\Ext(A)$ of all extreme points of~$A$: 
\begin{enumerate}
\item $\Ext^*(A)$ is the set of all extreme points of $A$ which are points of continuity of the identity map 
$(A,\ \text{weak-$*$-topology})\to(A,\ \text{norm topology})$.  
\item $\Dent^*(A)$ the set of \emph{weak-$*$-denting} points of $A$ 
and consists of the points $a\in A$ such that 
for every $\epsilon>0$ we have $a\not\in\oline{\conv}^{w^*}(\{b\in A\: \Vert a-b\Vert\ge\epsilon\})$.  
\item $\StrExp^*(A)$ is the set of \emph{weak-$*$-strongly exposed} points of $A$. 
It consists of the points $a\in A$ 
for which there exists a weak-$*$-continuous linear functional $f\colon\cX'\to\R$ such that  
$f(a)=\sup f(A)$ and for every sequence $\{a_n\}_{n\ge1}$ in $A$ 
with $\lim\limits_{n\to\infty}f(a_n)=f(a)$ we have $\lim\limits_{n\to\infty}\Vert a_n-a\Vert=0$. 
\end{enumerate}
It is easily seen that 
\begin{equation}\label{ext_chain}
\StrExp^*(A)\subseteq\Dent^*(A)\subseteq\Ext^*(A)\subseteq\Ext(A).
\end{equation}
\end{defn}

\begin{rem}\label{KM}
We recall that a Banach space is said to have the \emph{Kre\u\i{}n--Milman property} if every bounded, closed, convex subset is the (norm-)closed convex hull of its extreme points. 
If the Banach space under consideration is the topological dual of another Banach space, then it is known that the Kre\u\i n--Milman property is equivalent 
to the so-called Radon--Nikod\'ym property, which is further equivalent 
to the property that every weak-$*$-compact convex subset is the weak-$*$-closed convex hull of its weak-$*$-strongly exposed points;  
see \cite[Thm.~5.12]{Ph89}. 
In particular, if $A$ is weak-$*$-compact, convex, and nonempty, then in \eqref{ext_chain} we have 
$\StrExp^*(A)\neq\emptyset$. 
A description of the set of weak-$*$-denting points 
is provided by \cite[Thm. 1.3]{OP08}.

For later use, we also mention that the Kre\u\i n--Milman property is shared by both the ideal of trace-class operators on a (not necessarily separable) Hilbert space and the space of absolutely summable families $\ell^1(J)$ with a not necessarily countable index set; 
see \cite[Lemma 2]{Chu81}, and \cite[Cor. 4.1.9]{Bog83} as well as \cite{Lin66}, respectively. 
\end{rem}

\begin{rem}\label{extstar}
If $\pi\colon G\to\U(\cH)$ is a bounded representation, then the momentum set $I_\pi$ is weak-$*$-compact and is the weak-$*$-closed convex hull of $\im(\Phi_\pi)$, hence we get $\Ext(I_\pi)\subeq \oline{\im(\Phi_\pi)}^{w^*}$ by Milman's theorem 
(see \cite[Ch. 2, \S 4, Prop. 4]{Bou07}). 

It then follows that $\Ext^*(I_\pi)$ is contained in the norm closure of 
$\im(\Phi_\pi)$, and in particular $\Ext^*(I_\pi)\subeq\Inm_\pi$. 
Thence we get $\Ext^*(I_\pi)\subeq\Inm_\pi\cap\Ext(I_\pi)$, hence eventually $\Ext^*(I_\pi)\subeq\Ext^*(\Inm_\pi)$.
\end{rem}

We now record a simple folklore lemma which will be needed 
in the proof of Corollary~\ref{answer}.

\begin{lem}\label{seq}
Let $\cX$ be a real Banach space and assume that the functionals  $\phi,\phi_1,\phi_2,\dots\in\cX'$ satisfy the conditions 
$\Vert\phi_n\Vert\le\Vert\phi\Vert$ for every $n\ge 1$ 
and 
$\lim\limits_{n\to\infty}\phi_n=\phi$ in the weak-$*$-topology. 
Then $\lim\limits_{n\to\infty}\Vert\phi_n\Vert=\Vert\phi\Vert$.
\end{lem}

\begin{prf}
The hypothesis implies 
$\limsup\limits_{n\to\infty}\Vert\phi_n\Vert\le\Vert\phi\Vert$. 
Therefore, if the conclusion fails to be true, then we must have 
$\liminf\limits_{n\to\infty}\Vert\phi_n\Vert<\Vert\phi\Vert$. 
Then there exist $\epsilon>0$ and integers $n_1<n_2<\cdots$ 
such that $\Vert\phi_{n_k}\Vert+2\epsilon<\Vert\phi\Vert$ for every $k\ge1$. 
It then follows that there exists $x_0\in\cX$ such that $\Vert x_0\Vert=1$ 
and $\Vert\phi_{n_k}\Vert+\epsilon<\vert\phi(x_0)\vert$ for every $k\ge1$. 
Consequently $\vert\phi_{n_k}(x_0)\vert+\epsilon<\vert\phi(x_0)\vert$ for every $k\ge1$, and then we cannot possibly have 
$\lim\limits_{n\to\infty}\phi_n(x_0)=\phi(x_0)$. 
The later equality is however ensured by the hypothesis that 
$\lim\limits_{n\to\infty}\phi_n=\phi$ in the weak-$*$-topology. 
This contradiction concludes the proof.
\end{prf}

\begin{prop}\label{equality}
If $\lambda\in\cP$ and  
$\Ext^*(I_{\pi_\lambda^{\cA}})\ne\emptyset$, then we have 
$$\cO_{\psi_\lambda^{\cA}}=\Ext^*(I_{\pi_\lambda^{\cA}})
\subeq \Ext^*(\Inm_{\pi_\lambda^{\cA}}).$$
\end{prop}

\begin{prf}
Theorem~\ref{flags2} ensures that 
$I_{\pi_\lambda^{\cA}}=\oline{\conv}^{w^*}(\cO_{\psi_\lambda^{\cA}})$, 
hence by Milman's theorem (\cite[Ch. 2, \S 4, Prop. 4]{Bou07}) we get 
$\Ext(I_{\pi_\lambda^{\cA}})\subeq \oline{\cO_{\psi_\lambda^{\cA}}}^{w^*}$. 
This further implies that 
$\Ext^*(I_{\pi_\lambda^{\cA}})\subeq\oline{\cO_{\psi_\lambda^{\cA}}}^{\bn}=\cO_{\psi_\lambda^{\cA}}$, 
where the latter equality follows by Theorem~\ref{flags2}(ii). 

On the other hand, it is easily seen that the set $\Ext^*(I_{\pi_\lambda^{\cA}})$ is invariant under the coadjoint action of $\U(\cA)$. 
Therefore, if $\Ext^*(I_{\pi_\lambda^{\cA}})\ne\emptyset$, it follows by the above inclusion relation that we actually have 
$\Ext^*(I_{\pi_\lambda^{\cA}})=\cO_{\psi_\lambda^{\cA}}$. 
The inclusion $\Ext^*(I_{\pi_\lambda^{\cA}})
\subeq \Ext^*(\Inm_{\pi_\lambda^{\cA}})$ was already noted in 
Remark~\ref{extstar}. 
\end{prf}

For the next statement we recall that the $C^*$-algebra $\cA$ is said to be \emph{scattered} if every positive linear functional on $\cA$ is the sum of a sequence of pure functionals. 

\begin{cor}\label{scattered}
If $\cA$ is a scattered $C^*$-algebra, then the following assertions hold: 
\begin{description}
\item[\rm(i)] If $\lambda\in\cP$, then 
$\cO_{\psi_\lambda^{\cA}}=\Ext^*(I_{\pi_\lambda^{\cA}})
\subeq \Ext^*(\Inm_{\pi_\lambda^{\cA}})$.
\item[\rm(ii)] If $\lambda,\mu\in\cP$, then 
$$I_{\pi_\lambda}=I_{\pi_\mu}
  \iff \pi_\lambda^{\cA}\simeq \pi_\mu^{\cA}.$$
\end{description}
\end{cor}

\begin{prf}
It follows by the main theorem of \cite{Chu81} that the topological dual 
of $\cA$ has the Radon--Nikod\'ym property, and then $\Ext^*(I_{\pi_\lambda^{\cA}})\ne\emptyset$ by Remark~\ref{KM}. 
Therefore (i) follows by Proposition~\ref{equality}. 

For (ii), note that the implication ``$\Leftarrow$'' is obvious. 
Conversely, if $I_{\pi_\lambda^{\cA}}=I_{\pi_\mu^{\cA}}$, then 
$\Ext^*(\Inm_{\pi_\lambda^{\cA}})=\Ext^*(\Inm_{\pi_\mu^{\cA}})$, 
hence (i) shows that $\cO_{\psi_\lambda^{\cA}}=\cO_{\psi_\mu^{\cA}}$.  
Now Theorem~\ref{flags2} ensures that $\pi_\lambda^{\cA}\simeq \pi_\mu^{\cA}$. 
\end{prf}

The next corollary provides an alternative proof for  Proposition~\ref{prop:mostbasic}, and for this reason we resume the corresponding notation. 
This statement should also be compared with \cite[Thm. X.4.1(ii)]{Ne00}, 
which in particular describes representations of finite dimensional Lie groups 
for which the exposed points of the momentum sets coincide with the extreme points and constitute a coadjoint orbit. 

\begin{cor}\label{answer}
If $\lambda\in\cP$, then  
\begin{equation}\label{answer_eq}
\cO_{D_\lambda}\simeq\StrExp^*(I_\lambda)=\Dent^*(I_\lambda)=\Ext^*(I_\lambda)\subseteq\Ext(I_\lambda),
\end{equation}
and the latter inclusion may be strict. 
Moreover, for $\lambda,\mu\in\cP$ we have $I_\lambda=I_\mu$ if and only if $\pi_\lambda\simeq\pi_\mu$. 
\end{cor}
 
\begin{prf}
Corollary~\ref{scattered} applies since  
the $C^*$-algebra $\cA=\C\1+ K(\cH)$ is scattered. 
Therefore, in view of Definition~\ref{ext_def} 
(see \eqref{ext_chain}), the equalities in \eqref{answer_eq} are obtained as soon as we have proved that $\cO_{D_\lambda}\subeq\StrExp^*(I_\lambda)$ 
(Corollary~\ref{scattered}). 
Since $\StrExp^*(I_\lambda)$ is $\U(\cH)$-invariant, it suffices to check that 
$D_\lambda\in\StrExp^*(I_\lambda)$. 

We now use the method of Remark~\ref{rem:2.7}. 
To this end, let $f\colon \Herm_1(\cH)\to\R$, $f(X)=\Tr(D_\lambda X)$, 
and let 
$\{T_n\}_{n\ge1}$ be a sequence in $I_\lambda$ ($\subeq\Herm_1(\cH)$) 
such that $\lim\limits_{n\to\infty}f(T_n)= f(D_\lambda)$. 
Since $\Vert T_n\Vert_2\le\Vert D_\lambda\Vert_2$ 
(see Remark~\ref{rem:2.7}), it then follows at once that 
$\lim\limits_{n\to\infty}\Vert T_n-D_\lambda\Vert_2=0$. 
We get in particular
$\lim\limits_{n\to\infty}T_n=D_\lambda$ in the weak operator topology. 

On the other hand, $I_\lambda$ is contained in the ball 
centered at $0\in \Herm_1(\cH)$ with radius $\Vert D_\lambda\Vert_1$, 
as a direct consequence of Theorem~\ref{flags2}(i). 
It is easily seen that the weak operator topology coincides with the weak-$*$-topology on any ball in $\Herm_1(\cH)$, and then 
$\lim\limits_{n\to\infty}T_n=D_\lambda$ in the weak-$*$-topology of 
$\Herm_1(\cH)$. 
By using Lemma~\ref{seq} we now get 
$\lim\limits_{n\to\infty}\Vert T_n\Vert_1=\Vert D_\lambda\Vert_1$. 

Now the conclusions of the above two paragraphs imply 
$\lim\limits_{n\to\infty}\Vert T_n-D_\lambda\Vert_1=0$ 
since the trace class has the Radon--Riesz property 
(also called the Kadec--Klee property or the Kadets--Klee property), 
i.e.,  every weakly convergent sequence $x_n \to x$ with $\|x_n\| \to \|x\|$ 
converges; 
see \cite{Ar81} and \cite{Si81} for the case of separable Hilbert spaces and \cite{Le90} for a stronger property in the general case. 

Thus $D_\lambda\in\StrExp^*(I_\lambda)$, 
and this completes the proof of \eqref{answer_eq}. 
To see that the final inclusion in \eqref{answer_eq} may be strict, 
note that if $0\le \lambda\in\cP$, 
then we have $0\in \Ext(I_\lambda)\setminus\cO_{D_\lambda}$. 
\end{prf}

\appendix 

\section{Schur--Weyl duality for infinite dimensional spaces} 
\mlabel{app:a}

In this appendix we collect some general remarks on the 
decomposition of $V^{\otimes k}$ under $\GL(V)$ for an infinite 
dimensional vector space $V$. 
In particular, we explain how 
this can be adapted to the decomposition of finite tensor 
products of Hilbert spaces $\cH$ under the action of $\U(\cH)$, resp., 
$\GL(\cH)$, which was considered in \cite{Se57}, \cite{Ki73}, \cite{Ol79}, and so on. 


Let $V$ be a complex vector space and 
$V^{\otimes k}$ the $k$th tensor power of $V$. 
Clearly, the product group $\GL(V) \times S_k$ acts on this space, 
and since $S_k$ is finite, $V^{\otimes k}$ is a semisimple $S_k$-module. 
We identify the set $\hat S_k$ of equivalence class of simple $S_k$-modules 
with the set $\Part(k)$ of partitions of $k$ 
and write $M^\lambda$ for the simple $S_k$-module 
corresponding to the partition $\lambda = (\lambda_1, \ldots, \lambda_n)$. 

\begin{rem} For any inclusion $V_1 \into V_2$ of vector spaces, 
we obtain an inclusion of $S_k$-modules 
$V_1^{\otimes k} \into V_2^{\otimes k}$. 
\end{rem}

Combining the preceding remark with \cite[Thm.~9.1.2]{GW98}, 
it follows that if $\dim V = \infty$, then 
each irreducible $S_k$-module occurs in $V^{\otimes k}$, and 
if \break $\dim V < \infty$, then all modules corresponding to 
partitions $\lambda \in \Part(k,n)$ into at most $n = \dim V$ pieces 
occur. 

For $\lambda \in \Part(k)$, let $P_\lambda \in \C[S_k]$ denote 
the corresponding central projection, so that 
$P_\lambda V^{\otimes k}$ is the isotypic component of type 
$\lambda$. From finite dimensional Schur--Weyl Theory,  
we know that for each finite dimensional subspace 
$F \subeq V$, the space 
$P_\lambda(F^{\otimes k})$ is an isotypic $\GL(F)$-module 
whose multiplicity space is an $S_k$-module isomorphic 
to $M^\lambda$. Moreover, the $S_k$-multiplicity space 
$$ \bS_\lambda(F) := \Hom_{S_k}(M^\lambda, F^{\otimes k}) $$
is a simple $\GL(F)$-module, where $\GL(F)$ acts by composition in the 
range. Further, the evaluation map induces an isomorphism 
$$ \bS_\lambda(F) \otimes M^\lambda \to P_\lambda(F^{\otimes k}) $$
of $\GL(F) \times S_k$-modules. 

\begin{thm} $\bS_\lambda(V) := \Hom_{S_k}(M^\lambda, V^{\otimes k})$ is an irreducible 
$\GL(V)$-module. 
\end{thm}

\begin{prf} We show that each 
non-zero element $0 \not= a \in \bS_\lambda(V)$ is a cyclic vector. 
Since $M^\lambda$ is finite dimensional, $\im(a) \subeq V^{\otimes k}$ 
is finite dimensional, hence contained in $F^{\otimes k}$ for some 
finite dimensional subspace $F \subeq V$. This means that 
$a \in \bS_\lambda(F)$. 

For any finite dimensional subspace $E \subeq V$ containing $F$, 
the space $\bS_\lambda(E)$ is a simple $\GL(E)$-module, and 
since every element of $\GL(E)$ extends to an element of 
$\GL(V)$, it follows that 
$\Spann(\GL(V)a) \supeq \bS_\lambda(E)$. Since $E$ was arbitrary 
and $\bS_\lambda(V)$ is the union of all $\bS_\lambda(E)$, 
$a$ is  a cyclic vector in $\bS_\lambda(V)$. 
This proves that $\bS_\lambda(V)$ is irreducible. 
\end{prf}

\begin{cor} Under the action of $\GL(V) \times S_k$, we have 
the following decomposition of 
$V^{\otimes k}$ into simple submodules 
$$ V^{\otimes k} \cong \bigoplus_{\lambda \in \Part(k)} 
\bS_\lambda(V) \otimes M^\lambda. $$
\end{cor} 

\begin{rem} We obtain the same decomposition with respect 
to the smaller group $\GL_f(V)$ of all invertible 
linear maps $g$ for which $g - \id_V$ is of finite rank. 
The same arguments as for $\GL(V)$ apply. 
\end{rem}

The next statement is well known (its version for topological groups 
can be found for instance in \cite{KS77} or \cite{Ol90}). 
We include its proof for the sake of completeness. 

\begin{prop}\label{inductive}
Let $S$ be a topological involutive semigroup and $(S_i)_{i\in I}$ 
a directed family of involutive subsemigroups of $S$ such that 
$\bigcup_{i\in I}S_i$ is dense in $S$.
Also let $\cH$ be a complex Hilbert space and $(\cH_i)_{i\in I}$ a 
directed family of closed subspaces of $\cH$ such that 
$\bigcup_{i\in I} \cH_i$ is dense. 
Assume that $\rho\colon S\to B(\cH)$ is a strongly continuous 
$*$-representation of $S$ such that 
for each $i\in I$
the subspace $\cH_i$ is invariant invariant under $\rho(S_i)$ 
and the representation 
$\rho_i\colon S_i\to B(\cH_i)$, $g\mapsto \rho(g)|_{\cH_i}$,
is irreducible.  
Then the representation $\rho$ is irreducible. 
\end{prop}

\begin{prf}
It follows at once that the subset $\rho(S)$ of $ B(\cH)$ is self-adjoint. 
Thus its commutant $\rho(S)'$ is a von Neumann subalgebra of $ B(\cH)$, 
and what we have to prove is that $\rho(S)'=\C\1$. 

To this end, denote by $p_i\in B(\cH)$ the orthogonal projection on $\cH_i$ 
for all $i\in I$. 
Then $p_i\nearrow\1$ in the strong operator topology for $i\in I$. 
Now let $a\in\rho(S)'$ arbitrary and fix $i\in I$ for the moment. 
Then $a\in\rho(S_i)'$ and $p_i\in\rho(S_i)'$, 
so that $p_ia|_{\cH_i}\in\rho_i(S_i)'$. 
Since $\rho_i$ is an irreducible representation, 
we have $\rho_i(S_i)'={\mathbb C}\1\subseteq B(\cH_i)$, 
so there exists $z_i\in{\mathbb C}$ with $p_ia|_{\cH_i}=z_i\1\in B(\cH_i)$. 
Consequently $p_iap_i=z_ip_i$. 
If we let $n$ run through the nonnegative integers, 
we thus get a family of complex numbers $\{z_i\}_{i\in I}$. 

On the other hand, for all $i,j\in I$ with $i\le j$ we have $p_i\le p_j$, 
that is, $p_ip_j=p_jp_i=p_i$,
hence 
\begin{equation*}
z_ip_i=p_iap_i=p_ip_jap_jp_i=z_jp_ip_jp_i=z_jp_i.
\end{equation*}
This proves that $z_i=z_j$ whenever $i\le j$. 
Since $I$ is a directed set, 
it follows that actually $z_i=z_j$ for all $i,j\in I$, 
hence there exists $z_0\in\C$ such that $z_i=z_0$ for all $i\in I$. 
Therefore  
$p_iap_i=z_0p_i$ for every $i\in I$. 
Since $p_i\nearrow\1$ in the strong operator topology for $i\in I$, 
it follows that $a=z_0\1$. 
Thus $\rho(S)'={\mathbb C}\1$, and this completes the proof.
\end{prf}

\begin{rem} 
Let $\cH$ be a Hilbert space and 
$\cH^{\otimes k}$ be the Hilbert space $k$-fold tensor product. 
Then the action of the group $\U(\cH) \times S_k$ on this Hilbert 
space is unitary. 

If $F \subeq \cH$ is a finite dimensional subspace, then 
$\bS_\lambda(F) \subeq F^{\otimes k}$ is a simple 
$\U(F)$-module because $\GL(F) \cong \U(F)_\C$ and 
it is simple under $\GL(F)$. 
Next note that the union of all the subspaces $\bS_\lambda(F)$ of 
$\bS_\lambda(\cH) := \Hom_{S^k}(M^\lambda,\cH^{\otimes k})$ 
is a dense subspace, and that each unitary operator 
$u \in \U(F)$ extends to some element of $\U(\cH)$. 
Therefore Proposition~\ref{inductive} implies that the 
representation of $\U(\cH)$ on $\bS_\lambda(\cH)$ 
is irreducible. 

The same argument also implies that the representation 
of the subgroup $\U_\infty(\cH) := \U(\cH) \cap (\1 + K(\cH))$ 
on $\bS_\lambda(\cH)$ is irreducible. 
If $\cA \subeq B(\cH)$ is a unital $C^*$-algebra acting irreducibly 
on $\cH$, then it also follows from Theorem~\ref{thm:1.3} 
that the representation of $\U(\cA)$ on  $\bS_\lambda(\cH)$ is irreducible. 
\end{rem}

\section{Weyl group orbits and their convex hulls} \mlabel{app:b}

In this section we discuss the convex geometry of the 
weak-$*$-closed convex hull $\co(\lambda)$ of $\cW\lambda$ in 
$\ell^1(J,\R)$ and of the norm-closed convex hull $\co^\bn(\lambda)$ 
which is a subset of $\co(\lambda)$. 

We start with a general observation concerning finite reflection 
groups. 
\begin{lem}\label{lemma_app2} 
Let $\cW \subeq \GL(V)$ be a finite reflection 
group whose reflections have the form 
$$ s_\alpha(v) = v - \alpha(v)\check \alpha, \quad 
\alpha \in V^*, \check\alpha \in V, \alpha(\check \alpha)= 2, 
\alpha \in \Delta. $$
Then, for each $\lambda \in V^*$, we have 
$$ \cW\lambda \subeq \lambda - C_\lambda 
\quad \mbox{ for } \quad 
C_\lambda := \cone(\{ \alpha \in \Delta \: \lambda(\check \alpha) > 0\}) $$
and the function $f \: \cW \to \R, w \mapsto (w\lambda)(x)$ is maximal 
in $\1$ if and only if $x \in C_\lambda^\star$. 
\end{lem}

\begin{prf} This is a refinement of \cite[Prop.~V.2.7]{Ne00}, which 
asserts that for each chamber $C \subeq V$ with 
$$\lambda \in \check C^\star, \quad 
\check C = \cone\big(\{ \check \alpha \: \alpha(C) \subeq \R_+\}\big), $$
we have 
$\cW\lambda \subeq \lambda - C^\star.$
From this relation we immediately derive that 
$$ \cW\lambda 
\subeq \lambda - \bigcap_{\lambda \in \check C^\star} C^\star  
=  \lambda - \Big(\bigcup_{\lambda \in \check C^\star} C\Big)^\star.$$

Let $\Delta_C := \{ \alpha \in \Delta \: \alpha(C) \subeq \R_+\}$ denote 
the positive system defined by the chamber $C$. 
Then $\lambda \in \check C^\star$ is equivalent to 
$$ \lambda(\check \alpha) \geq 0 \quad \mbox{ for } \quad 
\alpha \in \Delta_C.$$  
Now $\Delta_\lambda := \{ \alpha \in \Delta \: \lambda(\check \alpha) 
\geq 0\}$ is a parabolic system of roots and 
the condition above is equivalent to 
$\Delta_C \subeq \Delta_\lambda$. 

Since each chamber $\check C^\star \subeq V^*$ is a fundamental 
domain for the action of $\cW$ on $V^*$ and $\cW$ acts simply 
transitive on the set of chambers (\cite{Hu92}), the 
set of all chambers containing $\lambda$ coincides with the 
orbit of $\check C_0^\star$ for a fixed chamber $C_0$ satisfying 
this condition under the stabilizer group 
$\cW_\lambda$. 
Accordingly 
$$ \bigcup_{\lambda \in \check C^\star} C 
= \cW_\lambda \oline{C_0}, $$
and the dual cone of this set is spanned by all roots 
in 
$$ \bigcap_{w \in \cW_\lambda} w\Delta_{C_0} 
= \{ \alpha \in \Delta_{C_0} \: \cW_\lambda \alpha \subeq \Delta_{C_0}\}.$$
If $\alpha \in \Delta_{C_0}$ satisfies $\lambda(\check \alpha)=0$, 
then $s_\alpha \in \cW_\lambda$ and $s_\alpha \alpha = -\alpha\not\in \Delta_{C_0}$ implies that 
$$ \bigcap_{w \in \cW_\lambda} w\Delta_{C_0} 
\subeq \{ \alpha \in \Delta_{C_0} \: \lambda(\check \alpha) > 0 \} $$
and the converse inclusion holds trivially. We conclude that 
$$ \cW\lambda \subeq \lambda - \cone\big(\{ \alpha \in \Delta \: 
\lambda(\check \alpha) > 0 \}\big) = \lambda - C_\lambda. $$

This implies that for each $x \in C_\lambda^\star$, the function 
$f$ is maximal in $w = \1$. If, conversely, 
this is the case, then 
$$ f(s_\alpha) = (s_\alpha\lambda)(x) 
=\lambda(x) - \lambda(\check\alpha)\alpha(x) \geq \lambda(x) $$
implies that $\alpha(x) \geq 0$ for $\lambda(\check \alpha) > 0$, i.e., 
$x \in C_\lambda^\star$. 
\end{prf}

Now we turn to the action of $\cW = S_{(J)}$ on the Banach space 
$\ell^1(J,\R)$. We define for $k \in \N$ and  $\mu \in \ell^1(J,\R)$
\[ L_k(\mu) := \sup \{ \mu_{j_1} + \cdots + \mu_{j_k} \: j_i \in J, 
|\{ j_1,\ldots, j_k\}| = k\}.\] 
It follows immediately from the definition that $L_k$ 
is an $S_J$-invariant weak-$*$ lower semicontinuous convex function 
satisfying $L_k(\mu) \leq \|\mu\|_1$. 
From \cite[Lemma~2.3]{Neu99} we recall that 
\begin{equation}\label{lkpm}
 L_k(\lambda) = L_k(\lambda_+) \quad \mbox{ and } \quad 
 L_k(-\lambda) = L_k(\lambda_-),
\end{equation}
which further implies that 
\[ S(\lambda) := \sum_j \lambda_j 
= S(\lambda_+) -S(\lambda_-) 
= \lim_{k \to \infty} L_k(\lambda)
- \lim_{k \to \infty} L_k(-\lambda).\] 
For $\lambda \in \ell^1(J,\R)$, 
let $\co(\lambda)$ denote the weak-$*$-closure of $\cW\lambda$ 
and $\co^\bn(\lambda)$ denote the norm closure of this set.

\begin{lem}\mlabel{lem:convhul} 
For $\lambda,\mu \in \ell^1(J,\R)$ we consider the following conditions: 
\begin{description}
\item[\rm(I1)] $L_k(\mu) \leq L_k(\lambda)$ for $k \in \N$. 
\item[\rm(I2)] $L_k(-\mu) \leq L_k(-\lambda)$ for $k \in \N$. 
\item[\rm(I3)] $\sum_j \mu_j = \sum_j \lambda_j$. 
\end{description} 
Then the following assertions hold: 
\begin{description}
\item[\rm(i)] $\co(\lambda)$ consists of all 
elements $\mu$ satisfying {\rm(I1/2)}.
\item[\rm(ii)] $\co^\bn(\lambda)$ consists of all 
elements $\mu$ satisfying {\rm(I1)-(I3)}.
\item[\rm(iii)] $\mu \in \co(\lambda) \Leftrightarrow 
\mu_\pm \in \co(\lambda_\pm)$. 
\item[\rm(iv)] $0 \leq \mu' \leq \mu, \mu \in \co(\lambda) \Rarrow 
\mu' \in \co(\lambda).$
\item[\rm(v)] $\co(\lambda) = \co(\lambda_+) - \co(\lambda_-).$
\item[\rm(vi)] $\Ext(\co(\lambda)) \subeq 
\Ext(\co(\lambda_+)) - \Ext(\co(\lambda_-)).$ 
\end{description}
\end{lem}

\begin{prf} (i) Since the functionals $L_k$ are weak-$*$ lower  
semicontinuous and convex, $\mu \in \co(\lambda)$ 
implies (I1/2). 
From \cite[Prop.~2.8(2)]{Neu99} we derive that 
(I1/2) imply that $\mu$ is contained in the 
$\|\cdot\|_\infty$-closure of $\conv(\cW\lambda)$, 
hence in particular contained in $\co(\lambda)$. 

(ii) If $\mu \in \co^\bn(\lambda)$, then $\mu$ satisfies 
(I3), in addition to (I1/2), because the 
summation functional $S(\mu) := \sum_j \mu_j$ 
is invariant under~$\cW$. 
The converse follows from \cite[Lemma~2.5]{Neu99}. 
Actually Neumann assumes that $J$ is countable, but since 
his result can be applied to the countable subset  
$\supp(\mu) \cup \supp(\mu)$, the assertion holds in general. 

(iii) follows immediately from (i), 
$L_k(\mu) = L_k(\mu_+)$ and $L_k(-\mu) = L_k(\mu_-)$. 

(iv) is a consequence of (iii). 

(v) Any $\mu \in \co(\lambda)$ can be written as 
$\mu = \mu_+ - \mu_-$, and we have seen in (iii) that 
$\mu_\pm \in \co(\lambda_\pm)$. 
To verify $\supeq$, it suffices to show that 
$\lambda_+ - \cW\lambda_- \subeq \co(\lambda)$ because 
$\co(\lambda)$ is $\cW$-invariant. For $w \in \cW$ we observe that 
\[ 0 
\leq (\lambda_+ - w\lambda_-)_+ 
\leq \lambda_+, \] 
so that (iv) implies $(\lambda_+ - w \lambda_-)_+ \in 
\co(\lambda_+)$. We likewise derive from 
\[ 0 
\leq (\lambda_+ - w\lambda_-)_- 
\leq w\lambda_- \]
that $(\lambda_+ - w \lambda_-)_- \in 
\co(w\lambda_-) = \co(\lambda_-)$. 
Hence $\lambda_+ - w \lambda_- \in \co(\lambda)$ follows from (iii). 

(vi) If $\mu \in \Ext(\co(\lambda))$, then (iii) and (v) 
immediately imply that 
$\mu_\pm \in \Ext(\co(\lambda_\pm))$. 
\end{prf}

\subsection*{Extreme points} 

Now we turn to extreme points. 
Let $\lambda\in\cP$ throughout the following statements. 

\begin{lem} \mlabel{lem:exposed} 
$\lambda$ is an exposed point of $\co(\lambda)$ and in 
particular $\lambda \in \Ext(\co(\lambda))$. 
\end{lem}

\begin{prf} Every $w \in \cW = S_{(J)}$ is a finite product 
of transpositions $s_\alpha$ for roots 
$\alpha = \eps_i - \eps_j \in \Delta$. Applying 
Lemma~\ref{lemma_app2} to the subspace 
$V \subeq \R^{(J)}$ spanned by the vectors 
$e_i - e_j$ for which $s_{\eps_i - \eps_j}$ occurs in $w$, 
it follows that 
\[ \cW \lambda \subeq \lambda  - C_\lambda \quad \mbox{ for } \quad 
C_\lambda = \cone\{ \eps_i - \eps_j \: \lambda_i > \lambda_j\}.\] 
The set 
\[ B := \oline{\conv}^{w^*}(\{ \eps_i - \eps_j \: \lambda_i > \lambda_j\}) 
\subeq \ell^1(J,\R) \] 
is weak-$*$-compact, and for each pair $(i,j)$ with 
$\lambda_i \geq \lambda_j$ we obtain for 
$x_\lambda := \sum_j \lambda_j e_j$ 
\begin{equation}\label{newno1} 
\la \eps_i - \eps_j, x_\lambda \ra = \lambda_i - \lambda_j \geq 1.
\end{equation} 
We conclude that the convex cone 
$\R_+ B$ is weak-$*$-closed with the compact base $B$ (\cite[Ch. II, \S 7, No. 3]{Bou07}). 
Hence 
\[ 
\co(\lambda) \subeq \lambda - \R_+ B,
\]  
so that 
\begin{equation}\label{newno2} 
\{\lambda\} = \{ \mu \in \co(\lambda) \: \mu(x_\lambda) = \max \la 
\co(\lambda), x_\lambda \ra\}.
\end{equation} 
This proves that $\lambda$ is an exposed point of $\co(\lambda)$. 
\end{prf}

\begin{rem}\mlabel{rem:wexp} 
In the notation used in the proof of Lemma~\ref{lem:exposed}, 
it follows by inequality~\eqref{newno1} that for every $b\in B$ 
we have 
\[(\forall b\in B)\quad \la b, x_\lambda \ra\ge1.\] 
Now let us consider the weak-$*$-continuous functional
$f:=\la\cdot,x_\lambda\ra\colon\ell^1(J,\R)\to\R$.  
Then we get by the above inequality along with \eqref{newno2} 
that $f(\lambda)=\sup f(\co(\lambda))$ and for every sequence $\{\mu_n\}_{n\ge1}$ in $\co(\lambda)$ 
with $\lim\limits_{n\to\infty}f(\mu_n)=f(\lambda)$ we have 
$\mu_n=\lambda-t_nb_n$ with $b_n\in B$ and $t_n\in\R_+$ satisfying 
\[
0\leq t_n\leq t_n\la b_n,x_\lambda\ra 
=\la\lambda,x_\lambda\ra- \la\mu_n,x_\lambda\ra
=f(\lambda)-f(\mu_n)\]
for every $n\ge1$. 
Hence 
$\lim\limits_{n\to\infty}t_n=0$, and then $\lim\limits_{n\to\infty}\Vert \mu_n-\lambda\Vert=0$.
This shows that we actually have $\lambda\in\StrExp^*(\co(\lambda))$ 
(see Definition~\ref{ext_def}). 
\end{rem}

Since $\co(\lambda)$ is a bounded, hence weak-$*$-compact subset, 
\cite[Cor.~to Prop.~2 in \S I.7.1]{Bou07} implies that 
\begin{equation}
\Ext(\co(\lambda)) \subeq \oline{\cW\lambda}^{w^*}.
  \label{eq:extp} 
\end{equation}
On every bounded subset of $\ell^1(J,\Z)$, 
the weak-$*$-topology coincides with the product topology 
induced from $\Z^J$. For a subset  $F \subeq J$ 
and $\mu \in \Z^J$, we define $\mu_F \in \Z^J$ by 
\[ (\mu_F)_j := \begin{cases}
\mu_j  & \text{ for } j \in F \\ 
0 & \text{ else}. \\ 
\end{cases}\] 
We thus obtain 
\begin{equation}
  \label{eq:orbclos}
\oline{\cW\lambda}^{w^*}
= \{ w\lambda_F \: w \in \cW, F\subeq J, |F| < \infty\}.
\end{equation}
The set of all these elements can also be specified by the 
condition 
\[ (\forall n \in \Z)\  |\mu^{-1}(n)| 
\leq |\lambda^{-1}(n)|.\] 
We derive in particular that 
\begin{equation}
  \label{eq:posneg}
\lambda_\pm \in \oline{\cW\lambda}^{w^*}.
\end{equation}

Now the interesting question is, for which subset 
$F \subeq J$ is $\lambda_F$ an extreme point of $\co(\lambda)$. 
For $0 \leq \lambda$, we  call a restriction 
$\lambda_F$, $F \subeq \supp(\lambda)$, an {\it upper part} of $\lambda$ if 
$j \not\in F$ implies that $k \not\in F$ whenever 
$0 < \lambda_k < \lambda_j$. This means that, for $k := |F|$, the 
set  $\lambda(F)$ contains the $k$-largest values of $\lambda$. 

\begin{lem}
  \mlabel{lem:extreme} 
  \begin{description}
\item[\rm(i)]  If $\lambda_+ \not=0 \not=\lambda_-$, then 
$0\not\in \Ext(\co(\lambda))$. 
\item[\rm(ii)] For every extreme point $\mu \in \co(\lambda)$ 
either $\mu_+ \in \cW\lambda_+$ or 
$\mu_- \in \cW\lambda_-$. 
\item[\rm(iii)] If $\mu_+ = \lambda_+$ and 
$\mu_- \in \Ext(\co(\lambda_-))$ or 
$\mu_- = \lambda_-$ and $\mu_+ \in \Ext(\co(\lambda_+))$, then 
$\mu \in \Ext(\co(\lambda))$. 
\item[\rm(iv)] If $0 \leq \lambda$, then 
$\lambda_F \in \Ext(\co(\lambda))$ if and only if 
it is an upper part of $\lambda$. 
\end{description}
\end{lem}

\begin{prf} (i) If $\lambda_+ \not=0 \not=\lambda_-$, then the description 
of $\co(\lambda)$ in Lemma~\ref{lem:convhul}(i) implies the existence 
of some $\delta > 0$ with $\pm \delta \eps_j \subeq \co(\lambda_\pm) \subeq 
\co(\lambda)$. Hence 
$0\not\in \Ext(\co(\lambda))$. 

(ii) Suppose that $\mu \in \Ext(\co(\lambda))$. We have 
already seen in Lemma~\ref{lem:convhul}(vi) that this implies 
\[ \mu_\pm \in \Ext(\co(\lambda_\pm)) 
\subeq \oline{\cW\lambda_\pm}^{w^*}.\] 
We therefore have 
$\mu_+ \in \cW \lambda_{F_1,+}$ and 
$\mu_- \in \cW \lambda_{F_2,-}$ for finite subsets 
$F_1 \subeq \supp(\lambda_+)$ and 
$F_2 \subeq \supp(\lambda_-)$. 
Assume, contrary to (ii), that 
$\lambda_{F_1,+} \not= \lambda_+$ and 
$\lambda_{F_2,-} \not= \lambda_-$. 
Then $\mu$ lies in the $\cW$-orbit of 
$\lambda_{F_1 \cup F_2} = \lambda_{F_1,+} - \lambda_{F_2,-}$. 
That this is not an extreme point follows from the fact that 
$0$ is not an extreme point of 
$\co(\lambda\res_{J \setminus (F_1 \cup F_2)})$, as we have seen in 
(i).

(iii)  Suppose that $\mu_+ = \lambda_+$ and write 
$\mu = t \alpha + (1-t) \beta$ with 
$\alpha, \beta \in \co(\lambda)$ and $0 < t < 1$.  
From the inequalities (I1) for 
$k \leq |\supp(\lambda_+)|$ it then follows that on 
$\supp(\lambda_+)$ both $\alpha$ and $\beta$ coincide 
with $\lambda_+$. The inequalities (I1) for $k > |\supp(\lambda_+)|$ 
further lead to 
$\lambda_+ = \alpha_+ = \beta_+$, and hence to 
$\mu_- = t \alpha_- + (1-t) \beta_-$. 
If $\mu_- \in \Ext(\co(\lambda_-))$, we thus arrive 
at $\alpha_- = \beta_- = \mu_-$, so that 
$\mu \in \Ext(\co(\lambda))$. 
The other assertion follows by replacing $\lambda$ by 
$-\lambda$. 

(iv) In view of \eqref{eq:extp} and \eqref{eq:orbclos}, it 
 remains to show that, for $\lambda \geq 0$ and 
$F \subeq \supp(J)$, the element 
$\lambda_F$ is extremal if and only if it is an upper part.
Since $\lambda$ is extreme by Lemma~\ref{lem:exposed}, 
we may w.l.o.g.\ assume that $\lambda_F \not=\lambda$, so that 
$L_k(\lambda_F) < L_k(\lambda)$ for some $k \in \N$. 
Let $k \in \N$ be minimal with this property, i.e., 
the $k-1$ largest values 
$\lambda_{j_1}, \ldots, \lambda_{j_{k-1}}$ 
of $\lambda$ and $\lambda_F$ coincide. In particular, 
$j_i \in F$ for $i \leq k-1$. By a similar argument as in 
(c), we see that, if 
$\lambda_F$ is extremal, then its restriction to 
$J' := J\setminus \{ j_1,\ldots, j_{k-1}\}$ 
is an extreme point of 
$\co(\lambda')$ for $\lambda' := \lambda\res_{J'}$. 
We may therefore assume that $k = 1$, i.e., that 
$L_1(\lambda_F) < L_1(\lambda)$. Since all 
other values of $\lambda_F$ are also values of $\lambda$, 
it follows that, for every $k \in \N$, 
\[ L_k(\lambda) - L_k(\lambda_F) \geq L_1(\lambda) - L_1(\lambda_F) > 0.\]  
If $\lambda_F \not=0$ and $\lambda_{j_0} > 0$, 
we obtain for 
$\delta < \min(L_1(\lambda) - L_1(\lambda_F), \lambda_{j_0})$ 
from Lemma~\ref{lem:convhul}(i) that 
$\lambda_F \pm \delta \eps_{j_0} \in \co(\lambda)$. 
Therefore $\lambda_F$ is not extremal. 
This proves that, whenever $\lambda_F$ is extremal, 
we must have $|F| = k-1$, so that $\lambda_F$ is an upper part. 

Suppose, conversely, that $\lambda_F$ is an upper part. 
Then a similar argument as in (iii), using that 
$0 \in \Ext(\co(\lambda))$ holds for  the restriction to $J\setminus F$, 
implies that $\lambda_F$ is extremal. 
\end{prf}

\begin{lem}\label{B5} 
We have 
\[ \co^\bn(\lambda) \cap \Ext(\co(\lambda)) = \cW\lambda,\] 
and if $\lambda \geq 0$ or $\lambda \leq 0$, then this is 
precisely the set of extreme points
\begin{equation}
  \label{eq:extnorm}
\Ext(\co^\bn(\lambda)) = \cW\lambda.
\end{equation}
\end{lem}

\begin{prf} 
First we note that Lemma~\ref{lem:convhul}(i),(ii) imply that 
\[ \co^\bn(\lambda) = \{ \mu \in \co(\lambda) \: 
S(\mu) = S(\lambda)\}. \] 
If $\lambda \geq 0$, then $S(\mu) \leq S(\lambda)$ for 
every $\mu \in \co(\lambda)$, 
so that $\co^\bn(\lambda)$ is a face of $\co(\lambda)$. 
Therefore 
\[ \Ext(\co^\bn(\lambda)) 
= \Ext(\co(\lambda)) \cap \co^\bn(\lambda).\] 
From the description of the extreme points of 
$\co(\lambda)$ in Lemma~\ref{lem:extreme}(iv) 
as elements conjugate to some upper part $\lambda_F$, 
we see that whenever $\lambda_F \not=\lambda$, then 
$S(\lambda_F) < S(\lambda)$. We thus arrive at \eqref{eq:extnorm}. 
The same argument applies for $\lambda \leq 0$. 

In general, we know that the extreme points $\mu$ of 
$\co(\lambda)$ are precisely the Weyl groups orbits of 
elements of theform 
$\lambda_+ - \lambda_{-,F}$, where $\lambda_{F,-}$ is an upper part of 
$\lambda_-$ and of $\lambda_{F,+} - \lambda_{-}$, where $\lambda_{F,+}$ 
is an upper part of $\lambda_+$ (Lemma~\ref{lem:extreme}). Since 
\[ S(\lambda_+ - \lambda_{-,F})
= S(\lambda_+)  - S(\lambda_{-,F})
> S(\lambda_+)  - S(\lambda_{-}) = S(\lambda) \] 
for $\lambda_{-,F} \not=\lambda_-$ and 
\[ S(\lambda_{+,F} - \lambda_{-})
< S(\lambda_+)  - S(\lambda_{-})= S(\lambda) \] 
for $\lambda_+ \not= \lambda_{F,+}$, every 
extreme point of $\co(\lambda)$ contained in 
$\co^\bn(\lambda)$ lies in $\cW\lambda$. 
This completes the proof. 
\end{prf}

\begin{prop}\mlabel{obw} 
For arbitrary $\lambda\in\cP$ we have 
\begin{equation}
  \label{eq:extnorm6}
\Ext(\co^\bn(\lambda)) = \co^\bn(\lambda) \cap \Ext(\co(\lambda)) =\cW\lambda.
\end{equation}
\end{prop}

\begin{prf} 
Due to Lemma~\ref{B5} we only have to prove that 
\begin{equation}\label{obw_proof_eq1}
\Ext(\co^\bn(\lambda))\subeq\cW\lambda. 
\end{equation}
First note that this inclusion is a direct consequence of 
\cite[Lemma I.19]{Ne98} if the index set $J$ is finite. 
It easily follows from this remark that if $\mu\in\Ext(\co^\bn(\lambda))$ and $\supp\mu$ is finite, then $\mu\in\cW\lambda$. 
Therefore, in order to prove that \eqref{obw_proof_eq1} holds true, 
it suffices to prove that the support of any extreme point of $\co^\bn(\lambda)$ 
is a finite set. 
 
To this end let $\mu\in\Ext(\co^\bn(\lambda))$ arbitrary 
and write $\mu=\mu_+-\mu_-$ with $\mu_\pm\ge0$ and 
$\supp\mu_+\cap\supp\mu_-=\emptyset$. 
There exists a partition $J=J_+\cup J_-$ with $\supp\mu_\pm\subset J_\pm$. 
If $J$ is infinite, then $J_+$ or $J_-$ is infinite. 
Without loss of generality, we assume that $J_+$ is infinite. 
If we denote by $\lambda_1\ge\dots\ge\lambda_N$ the positive values of $\lambda_+$ counted according to the multiplicities, 
then we may also assume that 
\[\supp\lambda_+=\{1,\dots,N\}\subset\N\subset J_+
\] 
and for a suitable subset $\widetilde{J}_+\subset J_+$ 
we have the partition $J_+=\N\dot\cup\widetilde{J}_+$. 

It follows by Lemma~\ref{lem:convhul}(ii) and \eqref{lkpm} that the 
property $\mu\in\co^\bn(\lambda)$ is equivalent to the conditions 
that for every integer $k\ge1$ we have $L_k(\mu_\pm)\le L_k(\lambda_\pm)$, 
and moreover $S(\mu)=S(\lambda)$.  
If we denote 
\[c:=S(\lambda_+)+S(\mu_-)-S(\lambda_-)\le S(\lambda_+),\]
and 
\[S=\{\nu\in \ell^1(J_+,\R):\ 0\le\nu,\ 
L_k(\nu)\le L_k(\lambda_+)\text{ for }k\ge1,\ 
S(\nu)=c\}\]
it then follows that $\mu_+\in S$. 

Note that we actually have $\mu_+\in\Ext(S)$. 
In fact, if $\mu_+$ were a nontrivial convex combination of two different elements in $S$, 
then we can extend these elements to functions on $J$ which are both equal to $\mu_-$ on $J_-$. 
These functions belong to $\co^\bn(\lambda)$ by 
Lemma~\ref{lem:convhul}(ii) and their corresponding convex combination is $\mu$, which is impossible. 

We now claim that 
$S=\co^\bn(\lambda')$, 
where $0\le \lambda'\in\ell^1(J_+,\R)=\ell_1(\N\cup\widetilde{J}_+,\R)$ 
is defined such that $\supp\lambda'\subeq\N$, and 
$\lambda'_k=a'_k-a'_{k-1}$, 
where $a'_k=\min(L_k(\lambda_+),c)$ for every $k\ge1$ and $a'_0=0$. 
Note that $a'_k\nearrow c$ as $k\to\infty$, 
and $\lambda'_1\ge\lambda'_2\ge\cdots\ge 0$ 
(the sequences $(a_k')$ and $(L_k(\lambda_+))$ are ``concave''). 
The set $\supp\lambda'$ is finite since so is $\supp\lambda_+$. 
Moreover we have $L_k(\lambda')=a'_k=\min(L_k(\lambda_+),c)$ for every $k\ge1$ and $S(\lambda')=c$, hence by using Lemma~\ref{lem:convhul}(ii) again 
we get $S=\co^\bn(\lambda')$, as claimed. 

It now follows that $\mu_+\in\Ext(\co^\bn(\lambda'))$ with $\lambda'\ge0$, 
and then $\supp\mu_+$ is a finite set as a consequence of Lemma~\ref{B5}. 
We may therefore choose a partition $J= J_+ \cup J_-$ in such a way that 
$J_-$ is infinite. Then a 
similar reasoning shows that $\supp\mu_-$ is likewise finite. 
Thus any extreme point $\mu\in\Ext(\co^\bn(\lambda))$ has finite support, 
and we are done, in view of the remarks at the very beginning of the proof. 
\end{prf}

\subsection*{Acknowledgment}
The present research was begun during a one-month visit of
first-named author
 at the Department of Mathematics of the Technische Universit\"at
Darmstadt,
and was further developed during the conferences held at the
Emmy-Noether-Zentrum in Erlangen and at the Mathematische
Forschungsinstitut
Oberwolfach; the hospitality of these institutes is gratefully
acknowledged.

\end{document}